\documentclass[11pt, reqno]{amsart}
\usepackage{mathrsfs}
\usepackage[active]{srcltx}
\usepackage{mathrsfs,amsmath}
\usepackage{mathtools}
\usepackage{longtable}
\usepackage{todonotes}
\usepackage[autostyle]{csquotes}
\allowdisplaybreaks

\newcommand{\abs}[1]{\left\lvert #1 \right\rvert}
\newcommand{\norm}[1]{\left\lVert #1 \right\rVert}

\newcommand{\RR }{\mathbb{R}}

\UseRawInputEncoding

\usepackage{xspace}
\usepackage{xcolor}
\usepackage[normalem]{ulem}
\usepackage{etoolbox}

\newtoggle{ignoreflag}

\definecolor{bluecite}{HTML}{0875b7}
\usepackage[unicode=true,
bookmarksopen={true},
pdffitwindow=true,
colorlinks=true,
linkcolor=bluecite,
citecolor=bluecite,
urlcolor=bluecite,
hyperfootnotes=false,
pdfstartview={FitH},
pdfpagemode= UseNone]{hyperref}

\newcommand{\ds}{\displaystyle}

\newtheorem{proposition}{Proposition}[section]
\newtheorem{theorem}{Theorem}[section]

\newtheorem{corollary}{Corollary}[section]

\newtheorem{remark}{Remark}[section]
\numberwithin{equation}{section}

\usepackage{bbm}  
\usepackage{dsfont}

\title[Hypercontractivity of heat flow  on ${\sf RCD}(0,N)$ spaces]{Hypercontractivity of the heat flow  on ${\sf RCD}(0,N)$ spaces: sharpness and rigidities}
\author{Shouhei Honda, Alexandru Krist\'aly, and Alexandru P\^irvuceanu}

\address{Graduate School of Mathematical Science, The University of Tokyo, 1 3-8-1 Komaba, Meguro-Ku, Tokyo 153-8914, 
	Japan
}

\email{shouhei@ms.u-tokyo.ac.jp}

\address{Department of Economics, Babe\c s-Bolyai University, Str. Teodor Mihali 58-60, 400591 Cluj-Napoca,
	Romania \& Institute of Applied Mathematics, \'Obuda
	University, B\'ecsi \'ut 96/B, 1034
	Budapest, Hungary}

\email{alexandru.kristaly@ubbcluj.ro; kristaly.alexandru@uni-obuda.hu}

\address{Faculty of Mathematics and Computer Science, Babe\c s-Bolyai University, 1 Mihail Kog\u alniceanu, 400591 Cluj-Napoca,
	Romania
}

\email{pirvuceanualexandrudaniel@gmail.com}

\thanks{S. Honda is  supported by the Grant-in-Aid for Scientific Research (A) of 25H00586, the Grant-in-Aid for Scientific Research (B) of 21H00977 and the Grant-in-Aid for Transformative Research Area (A) of 22H05105, 
	Japan.} 
\thanks{A.\ Krist\'aly is  supported by the
	Excellence Researcher Program \'OE-KP-2-2022 of \'Obuda University, Hungary.}

\begin{document}
	\begin{abstract}
		The main goal of the present paper is to provide  \textit{sharp} hypercontractivity bounds of the heat flow $({\sf H}_t)_{t\geq 0}$ on ${\sf RCD}(0,N)$ metric measure spaces. 
		The best constant in this estimate involves the  asymptotic volume ratio, and its optimality is obtained by means of the sharp $L^2$-logarithmic Sobolev inequality on ${\sf RCD}(0,N)$ spaces  and a blow-down rescaling argument. Equality holds in this sharp estimate for a prescribed time $t_0>0$ and a non-zero extremizer $f$ if and only if  the ${\sf RCD}(0,N)$ space has an $N$-Euclidean cone structure and 
		$f$ is a Gaussian whose dilation factor is reciprocal to $t_0$, up to a multiplicative constant. Applications include an extension of Li's rigidity result, almost rigidities, as well as topological rigidities of non-collapsed ${\sf RCD}(0, N)$ spaces. Our results are new even  on  complete Riemannian manifolds with non-negative Ricci curvature.
		%
		
	\end{abstract}
	\maketitle 
	
	\vspace{-0.4cm}
	
	\section{Introduction and Results}
	The heat flow describes the evolution of temperature in the space-time within an object immersed in an ambient space. 	It is well-known that the heat kernel/flow is strongly influenced by the  dimension of the ambient space and volume growth, encoding information on the curvature, see e.g.\ Grigor'yan \cite{Grigorian, Grigorian-2, Grigorian-book},  Li \cite{Li-Annals}, Ni \cite{Ni-JGA}, Saloff-Coste \cite{Saloff-Coste-2, Saloff-Coste}, etc. Such problems are mostly studied on complete Riemannian manifolds and, more widely, on Markovian triples, providing various asymptotic behaviors and hypercontractivity/ultracontractivity estimates for the heat kernel/flow,  see e.g.\ Bakry \cite{Bakry}, Ledoux \cite{Ledoux-1}, Bakry, Concordet, and Ledoux \cite{BCL}, Bakry, Gentil, and Ledoux \cite{BGL}.
	
	Given  a metric measure space $(X,{\sf d}, {\sf m})$, a central question is to provide a hypercontractivity estimate of the heat flow $({\sf H}_t)_{t\geq 0}$ on $X$: if $1\leq p\leq q\leq \infty$, find a non-increasing, possibly optimal, function $t\mapsto m_{p,q}(t)$ such that 
	\begin{equation}\label{Gross-relation}
		\|{\sf H}_t\|_{p,q}\leq e^{m_{p,q}(t)},\ \ t>0,
	\end{equation}
	where 
	$$\|{\sf H}_t\|_{p,q}=\sup_{\|f\|_{L^p(X,{\sf m})}\leq 1}\|{\sf H}_tf\|_{L^q(X,{\sf m})}$$ stands for the operator norm of ${\sf H}_t$ from $L^p(X,{\sf m})$ to $L^q(X,{\sf m})$. 
	
	The first estimates of the type \eqref{Gross-relation} were established by Glimm \cite{Glimm} and Simon and Hoegh-Krohn \cite{Simon-HK}, as well as by Nelson \cite{Nelson, Nelson2} for Ornstein--Uhlenbeck semigroups. Shortly after these results, a systematic study was initiated by Gross \cite{Gross, Gross2, Gross3}, who proved the equivalence   between the hypercontractivity estimate \eqref{Gross-relation} (for a suitable function $m_{p,q}$) and the classical logarithmic Sobolev inequality. After Gross' work, a kind of symbiosis arose between hypercontractivity and  logarithmic Sobolev inequalities, see e.g.\ Bobkov,  Gentil, and Ledoux \cite{BobGL}, Eldredge,  Gross, and Saloff-Coste \cite{EGSC}, Saloff-Coste \cite{Saloff-Coste-3}. In this way, Bakry \cite{Bakry} and  Bakry, Concordet, and Ledoux \cite{BCL}, 
	provided a general criterion to establish \eqref{Gross-relation} on Markov triples  satisfying a general Entropy-Energy inequality, i.e., a generalized logarithmic Sobolev inequality.

	In order to take a glance at these results, we recall the main hypercontractivity estimate from 
	Bakry, Gentil, and Ledoux \cite{BGL}; namely,  once an Euclidean-type logarithmic Sobolev inequality holds with the \textit{optimal} Euclidean constant, involving $n\in \mathbb N$ as a dimension (even for a generator ${\sf L}$ associated to a \textit{carr\'e du champ} $\Gamma$),  one has for every $1\leq p\leq q\leq \infty$ that
	\begin{equation}\label{Bakry-relation}
		\|{\sf H}_t\|_{p,q}\leq \frac{{\sf M}(p,q)^\frac{n}{2}}{(4\pi t)^{\frac{n}{2}(\frac{1}{p}-\frac{1}{q})}},\ \ \forall t>0,
	\end{equation}
	where  	
	$${\sf M}(p,q)=\frac{p^{\frac{1}{p}}\left(1-\frac{1}{p}\right)^{1-\frac{1}{p}}}{q^{\frac{1}{q}}\left(1-\frac{1}{q}\right)^{1-\frac{1}{q}}}\left(\frac{1}{p}-\frac{1}{q}\right)^{\frac{1}{p}-\frac{1}{q}},$$
	with the usual limiting conventions, i.e., ${\sf M}(p,p)=1$, ${\sf M}(1,q)=q^{-\frac{1}{q}}$, ${\sf M}(p,\infty)=\left(1-\frac{1}{p}\right)^{1-\frac{1}{p}}$, and ${\sf M}(1,\infty)=1.$ By using the semigroup property of the heat flow $({\sf H}_t)_{t\geq 0}$, one can prove that  the  bound  \eqref{Bakry-relation} is optimal. 
	In particular, by \eqref{Bakry-relation} one has the ultracontractivity estimate 
	\begin{equation}\label{kernel-relation}
		\|{\sf H}_t\|_{1,\infty}\leq \frac{1}{(4\pi t)^{\frac{n}{2}}},\ \ \forall t>0.
	\end{equation}
	By using a deep result of Li \cite{Li-Annals} -- who established the sharp asymptotic behavior of the heat kernel \nolinebreak --  inequality \eqref{kernel-relation}
	holds on an $n$-dimensional complete Riemannian manifold $(M,g)$ with non-negative Ricci curvature   if and only if $M$ is isometric to the Euclidean space $\mathbb R^n$; for details, see Bakry, Concordet, and Ledoux \cite[p.\ 399]{BCL}.   The latter rigidity clearly shows the subtle interconnection between the heat kernel/flow and the curvature of the ambient space, as we have already noticed. 
	
	Having these facts in mind, the purpose of the present paper is to establish sharp hypercontractivity estimates on ${\sf RCD}(0,N)$   spaces, $N>1$, to study their equality cases, and to obtain several rigidity, almost rigidity, and topological rigidity results; in fact, we expect that the geometry of such objects is reflected in the sharp hypercontractivity estimates. The ${\sf RCD}(0,N)$   spaces are the celebrated metric measure spaces verifying the curvature-dimension condition ${\sf CD}(0,N)$ in the sense of Lott--Sturm--Villani and having  infinitesimal Hilbertian structures, and they include complete Riemannian manifolds with non-negative Ricci curvature (having dimension at most $N$), their Gromov--Hausdorff limits with possible singularities,  Alexandrov spaces with non-negative
	curvature, etc.
	We emphasize that the theory of heat flows on ${\sf RCD}(0,N)$ spaces is well understood due to the seminal work of Ambrosio,  Gigli, and Savar\'e \cite{AGS} and the subsequent references. 
	
	In order to state our results, let $(X,{\sf d},{\sf m})$ be an   ${\sf RCD}(0,N)$ metric measure space, $N>1$, and let $({\sf H}_t)_{t\geq 0}$ be the heat flow on $(X,{\sf d},{\sf m})$, see Ambrosio,  Gigli, and Savar\'e \cite{AGS}.  Due to the generalized Bishop--Gromov principle, see Sturm \cite{Sturm-2},  the \textit{asymptotic volume ratio}, given by	$${\sf AVR}(X)=\lim_{r\to \infty}\frac{ {\sf m}(B_r(x))}{\omega_Nr^N},$$ is well-defined (i.e., does not depend on the point $x\in X$), 
	where $\omega_N=\pi^{N/2}/\Gamma(N/2+1)$ and $B_r(x)=\{y\in X:{\sf d}(x,y)<r\}$ is the metric ball with origin  $x\in X$ and radius $r>0$.   



	%
	%
	%

	By using the above notation, our first result states a sharp hypercontractivity estimate of the heat flow on ${\sf RCD}(0,N)$ spaces, comparable to \eqref{Bakry-relation},
	where the celebrated asymptotic volume ratio plays a crucial role:

	\begin{theorem}[Sharp hypercontractivity estimate]\label{main-theorem-RCD} Given $N>1$, let $(X,{\sf d},{\sf m})$ be an   ${\sf RCD}(0,N)$ metric measure space.   
		Then the following statements are equivalent$:$ 
		\begin{itemize}
			\item[(i)]  the heat flow verifies a large time hypercontractivity estimate for some $1 \le p<q \le \infty$, i.e., 
			$$
			\liminf_{t \to \infty}(4\pi t)^{\frac{N}{2}(\frac{1}{p}-\frac{1}{q})}\|{\sf H}_t\|_{p, q}<\infty; 
			$$
			\item[(ii)] ${\sf AVR}(X) >0$. 
		\end{itemize}
		In a quantitative form, if ${\sf AVR}(X) >0$ and  $1\leq p\leq q\leq \infty$, one has 
		\begin{equation}\label{Heat-bound-main}
			\norm{{\sf H}_t f}_{L^q(X,{\sf m})}\le \frac{{\sf M}(p,q)^\frac{N}{2}{\sf AVR}(X)^{\frac{1}{q}-\frac{1}{p}}}{(4\pi t)^{\frac{N}{2}(\frac{1}{p}-\frac{1}{q})}}\norm{f}_{L^p(X,{\sf m})},\ \ \forall t>0,\ \forall f\in L^1(X,{\sf m})\cap L^\infty(X,{\sf m}). 
		\end{equation}
		In other words,
		\begin{equation}\label{sharp operator norm}
			(4\pi t)^{\frac{N}{2}(\frac{1}{p}-\frac{1}{q})}\|{\sf H}_t\|_{p, q} \le {\sf M}(p,q)^\frac{N}{2}{\sf AVR}(X)^{\frac{1}{q}-\frac{1}{p}},
		\end{equation}
		and the latter inequality is sharp in the sense that
		\begin{equation}\label{sharp operator norm limit}
			\lim_{t \to \infty}(4\pi t)^{\frac{N}{2}(\frac{1}{p}-\frac{1}{q})}\|{\sf H}_t\|_{p, q} = {\sf M}(p,q)^\frac{N}{2}{\sf AVR}(X)^{\frac{1}{q}-\frac{1}{p}}, \quad \forall p  \le  \forall q.
		\end{equation}
		
	\end{theorem}
	
	Note that the corresponding result in the case when $N=1$ is well-known because any ${\sf RCD}(0, 1)$ space with ${\sf AVR}>0$ is $\mathbb{R}$ or $[0, \infty)$, up to multiplying the reference measure ${\sf m}$ by a positive constant, see Kitabeppu and Lakzian \cite{KL}.
	
	The non-collapsing feature of metric balls on metric measure spaces with Ricci curvature bounded from below (both in smooth and synthetic cases) characterizes the validity of various Sobolev inequalities, see e.g.\ Coulhon and Saloff-Coste \cite{CSC} on 
	Riemannian manifolds and Krist\'aly \cite[Theorem 5.6]{Kristaly-new} on ${\sf CD}(0, N)$ spaces; the first part of Theorem  \ref{main-theorem-RCD} perfectly fits into this framework in terms of the hypercontractivity of the heat flow on  ${\sf RCD}(0, N)$ spaces.
	
	The  proof of \eqref{Heat-bound-main} can be  deduced from Bakry \cite{Bakry}, see also  Bakry, Gentil, and Ledoux \cite[Theorem 7.1.2]{BGL}, combined with the  sharp $L^2$-logarithmic Sobolev inequality on ${\sf RCD}(0,N)$ metric measure spaces recently established by Balogh, Krist\'aly, and Tripaldi \cite{BKT}. Since we are  interested in studying the equality  in \eqref{Heat-bound-main}, we reconstruct its whole proof 
	in order to trace back the estimates. The sharpness (\ref{sharp operator norm limit}) is a consequence of the behavior of  $\|{\sf H}_t\|_{p, q}$ with respect to the pointed measured Gromov--Hausdorff convergence (shortly, pmGH convergence),  a suitable blow-down rescaling argument, together with an Euclidean cone rigidity established by De Philippis and Gigli \cite{DG0}.
	
	As a byproduct of the latter arguments, we can provide an elegant, short proof of the sharp Li-type large time behavior of the heat kernel in the non-smooth setting of ${\sf RCD}(0, N)$ spaces. This question has been investigated by Jiang \cite[Theorem 1.4]{Jiang} and  Jiang, Li, and Zhang \cite{JLZ}, where Harnack-type estimates played a crucial role. Our approach is based on a suitable blow-down rescaling, which also indicates the optimality of the asymptotic growth used by Li \cite[Theorem 1]{Li-Annals} (see relation \eqref{Li-growth})  in the setting of Riemannian manifolds with non-negative Ricci curvature. 

	

	Coming back to Theorem \ref{main-theorem-RCD}, the next result  deals with the equality case in  \eqref{Heat-bound-main}:
	
	\begin{theorem}[Equality in the hypercontractivity estimate]\label{theorem-equality-RCD} Given $N>1$,
		let $(X,{\sf d},{\sf m})$ be an   ${\sf RCD}(0,N)$ metric measure space with ${\sf AVR}(X) >0.$ 
		Then the following  statements are equivalent$:$ 
		\begin{itemize}
			\item[(i)] 	 
			Equality holds in \eqref{Heat-bound-main} for some  $1< p< q< \infty$,  $t_0>0$, and non-negative  non-zero function $f\in L^1(X, {\sf m})\cap L^\infty(X,{\sf m});$
			
			\item[(ii)]  $(X,{\sf d}, {\sf m})$ is isometric to the $N$-Euclidean cone over an ${\sf RCD}(N-2, N-1)$ space. 
		\end{itemize}
		Moreover, if one of them holds $($thus both of them hold$),$ then $f$ as in $({\rm i})$ can be determined uniquely, up to a multiplicative factor, by
		$f=e^{-\alpha_0 {\sf d}^2(x_0,\cdot)}$  for
		${\sf m}$-a.e.\ on $X,$ where $x_0$ is a tip of the $N$-Euclidean cone, and
		\begin{equation}\label{alpha-definitio}
			\alpha_0=\frac{1}{4t_0}\frac{p}{p-1}\left(\frac{1}{p}-\frac{1}{q}\right)>0.
		\end{equation}

				%
				%
				%
	\end{theorem}
		
		\begin{remark}\rm (i) Note that Theorem \ref{theorem-equality-RCD}  is in perfect concordance with the celebrated result of Lieb \cite[Theorem 4.5]{Lieb}, where the equality case in certain operator norm 
			estimates  (involving degenerate kernels) is characterized,   the extremals being the Gaussians; for a recent development, see Barthe and Wolff \cite{Barthe-Wolff}. 
			
			(ii)
			If a function $f \in L^1(X, {\sf m}) \cap L^{\infty}(X, {\sf m})$ (without assuming the non-negativity) attains the equality in \eqref{Heat-bound-main} , then $f^-, f^+$, and $|f|$ also attain the equality,  where $f^+=\max \{f, 0\}$ and $f^-=\max \{-f, 0\}$. Indeed, since  $|{\sf H}_tf| \le {\sf H}_t|f|$ holds for ${\sf m}$-a.e., we see that $|f|$ also attains the equality. Then writing  (\ref{Heat-bound-main}) for $f^{\pm}$ and then taking the sum shows that $f^{\pm}$ also attain the equality.
		\end{remark}

			As a direct consequence of the above results, we obtain the following rigidity with respect to the Euclidean space result in the Riemannian framework, which seems to be new even in the Euclidean setting.

			\begin{corollary}\label{corollary-equality-Riemannian-1}
				Let $(M,g)$ be an $n$-dimensional complete Riemannian manifold with non-negative Ricci curvature and ${\sf AVR}(M) >0$, endowed with its natural distance function ${\sf d}_g$ and canonical measure $v_g$. Then the following statements are equivalent$:$ 
				\begin{itemize}
					\item[(i)] 	 
					Equality holds in \eqref{Heat-bound-main} for some  $1< p< q< \infty$,  $t_0>0$, and non-negative, non-zero $f\in L^1(M,v_g)\cap L^\infty(M,v_g);$ 
					
					\item[(ii)]  $(M,g)$ is isometric to the Euclidean space $\mathbb R^n$ and for some $C>0$ and $x_0\in \mathbb R^n$, 
					$f(x)=Ce^{-\alpha_0|x-x_0|^2}$ for a.e.\ $ x\in \mathbb R^n,$ where $$\alpha_0=\frac{1}{4t_0}\frac{p}{p-1}\left(\frac{1}{p}-\frac{1}{q}\right)>0.$$ 
				\end{itemize}
			\end{corollary}
			
			The proof of Corollary \ref{corollary-equality-Riemannian-1} is  a simple consequence of Theorem \ref{theorem-equality-RCD}, because if a complete Riemannian manifold is a Euclidean cone, then it must be an Euclidean space. However, we will also provide an alternative proof of it based on Fourier analysis together with the fact that any point $x_0\in M\simeq \mathbb R^{n}$ is a tip, see \S \ref{section-altenative}. 
			
			Theorems \ref{main-theorem-RCD} and \ref{theorem-equality-RCD} can be efficiently used to state various   rigidity,  almost rigidity, as well as topological rigidity results  on non-collapsed ${\sf RCD}(0, N)$ spaces; see \S \ref{section-rigidity}. Here is such a result:

			\begin{theorem}[Li's rigidity on non-collapsed ${\sf RCD}(0,N)$ spaces]
				\label{Li-rigidity}	Let $(X, {\sf d}, \mathcal{H}^N)$ be a non-collapsed  ${\sf RCD}(0,N)$ space for some $N \in \mathbb{N}$. Then the following three statements are equivalent.
				\begin{enumerate}
					\item[(i)] For some $1 \le p<q \le \infty$, one has
					$$
					\|{\sf H}_t\|_{p, q} \le \frac{{\sf M}(p,q)^\frac{N}{2}}{(4\pi t)^{\frac{N}{2}(\frac{1}{p}-\frac{1}{q})}}, \quad \forall t>0;
					$$
					\item[(ii)] One has
					$$
					\|{\sf H}_t\|_{p, q} \le \frac{{\sf M}(p,q)^\frac{N}{2}}{(4\pi t)^{\frac{N}{2}(\frac{1}{p}-\frac{1}{q})}}, \quad \forall t>0, \quad 1\le \forall p \le \forall q \le \infty;
					$$
					\item[(iii)] $X$ is isometric to $\mathbb R^N.$
				\end{enumerate}
			\end{theorem}
			
			The paper is organized as follows. In Section \ref{section-2} we collect those notions and results that are necessary for the elaboration of our results, i.e., Sobolev spaces on metric measure spaces,  ${\sf RCD}(0,N)$ spaces, heat flows, and sharp logarithmic-Sobolev inequalities on ${\sf RCD}(0,N)$ spaces. Section \ref{section-3} is devoted to the proof of Theorem \ref{main-theorem-RCD} together with further fine properties of the heat kernel. In particular, in Theorem \ref{theorem-Jiang-Li}  we provide a genuinely new proof of the sharp large time behavior of the heat kernel of Li \cite{Li-Annals} in the setting of ${\sf RCD}(0,N)$ spaces by combining the blow-down rescaling argument with the pointed measured Gromov--Hausdorff convergence. In Section \ref{section-4} we deal with the equality case in our sharp hypercontractivity estimate by proving  Theorem \ref{theorem-equality-RCD}. An alternative proof of Corollary \ref{corollary-equality-Riemannian-1} is also provided by using Fourier analysis. Section \ref{section-rigidity}
			is devoted to proving various rigidity, almost rigidity, and topological rigidity results on ${\sf RCD}(0,N)$ spaces, see e.g. Theorem  \ref{Li-rigidity}. In addition, a topological result is also provided on Riemannian manifolds with non-negative Ricci curvature, which roughly says that if the deficit in the hypercontractivity estimate is closer and closer to the sharp Euclidean estimate, then the manifold is closer and closer topologically to the Euclidean space, this being formulated in terms of the trivialization of the higher-order homotopy groups, see  Theorem \ref{Perelman-Munn-theorem}.

			\section{Preliminaries} \label{section-2}
			
			In this section  we collect those notions and results that are indispensable in our further study.

			\subsection{Sobolev spaces on metric measure spaces}
			Let  $(X,{\sf d}, {\sf m})$ be a metric measure space, i.e.,
			$(X,{\sf d})$ is a complete separable metric space and 
			$ {\sf m}$ is a locally finite measure on $X$ endowed with its
			Borel $\sigma$-algebra. Throughout the paper, we assume that supp$({\sf m})=X$. 
			
			For  every $p \geq 1$, we denote by $$L^p(X,{\sf m})=\left\{ u\colon X\to\mathbb{R}: u\text{ is measurable, }\displaystyle\int_X\vert u\vert^p\,{\rm d}{\sf m}<\infty\right\}$$ the set of $p$-integrable functions on $X$; this space is endowed with the norm $\|\cdot\|_{L^p(X,{\sf m})}.$ For $p=\infty$, the supremum-norm $\|\cdot\|_{L^\infty(X,{\sf m})}$ applies. 
			
			Let ${\rm Lip}(X)$  (resp.\ ${\rm Lip}_{\rm loc}(\Omega)$) be the space of real-valued Lipschitz  (resp.\  locally Lipschitz) functions over $X$.  
			If $u\in {\rm Lip}_{\rm loc}(X)$,  the \textit{local Lipschitz constant} $|{\rm lip}_{\sf d} u|(x)$ of $u$ at $x\in X$  is 
			$$|{\rm lip}_{\sf d} u|(x)=\limsup_{y\to x}\frac{|u(y)-u(x)|}{{\sf d}(x,y)}.$$
			The \textit{Cheeger energy} ${\sf Ch}:L^2(X,{\sf m})\to [0,\infty]$  is defined as the convex 
			functional 
			\begin{equation}\label{Cheeger-definition}
				{\sf Ch}(u)=\inf\left\{\liminf_{n\to \infty} \int_X |{\rm lip}_{\sf d} u_n|^2 {\rm d} {\sf m}:(u_n)\subset {\rm Lip}(X)\cap L^2(X,{\sf m}),\ u_n\to u\ {\rm in}\ L^2(X,{\sf m})  \right\},
			\end{equation}
			see Cheeger \cite{Cheeger} and   Ambrosio, Gigli, and Savar\'e \cite{AGS}.   The natural Sobolev space over the metric measure space $(X,{\sf d}, {\sf m})$ is defined as 
			$$W^{1,2}(X,{\sf d},{\sf m})=\{u\in L^2(X,{\sf m}):{\sf Ch}(u)<\infty\},$$ endowed with the norm 
			$\|u\|_{W^{ 1,2}}=\left(\|u\|^2_{L^2(X,{\sf m})}+{\sf Ch}(u)\right)^{1/2}. $
			In general,  $W^{1,2}(X,{\sf d},{\sf m})$ is a Banach space. For  $u\in W^{1,2}(X,{\sf d},{\sf m})$, one may take a minimizing sequence $\{u_n\}$ in \eqref{Cheeger-definition}, obtaining an optimal $L^2$-function, denoted by $|\nabla u|$, which is called the minimal relaxed slope of $u$, such that  
			$${\sf Ch}(u)= \int_X |\nabla u|^2\, {\rm d} {\sf m}.$$
			

			\subsection{${\sf RCD}(0,N)$ spaces}
			Let $(X,{\sf d}, {\sf m})$ be a metric measure space, and let	$\mathcal P(X,{\sf d})$ be the
			\linebreak $L^2$-Wasser\-stein space of probability measures on $X$, while
			$\mathcal P(X,{\sf d}, {\sf m})$ is  the subspace of
			$ {\sf m}$-absolutely continuous measures on $X$.
			For  $N> 1,$ we consider ${\rm Ent}_N(\cdot| {\sf m}):\mathcal P(X,d)\to \mathbb R$, the {\it R\'enyi entropy functional} with
			respect to the measure $ {\sf m}$,
			defined by 
			\begin{equation}\label{entropy}
				{\rm Ent}_N(\nu| {\sf m})=-\int_X \rho^{-\frac{1}{N}}{\rm d}\nu=-\int_X \rho^{1-\frac{1}{N}}{\rm d} {\sf m},
			\end{equation}
			where $\rho$ is the density function of $\nu^{\rm ac}$ in
			$\nu=\nu^{\rm ac}+\nu^{\rm s}=\rho  {\sf m}+\nu^{\rm s}$, while $\nu^{\rm ac}$ and $\nu^{\rm s}$
			represent the absolutely continuous and singular parts of $\nu\in
			\mathcal P(X,{\sf d}),$ respectively.
			
			According to  Lott and Villani \cite{LV} and Sturm \cite{Sturm-2}, the \textit{curvature-dimension condition} ${\sf CD}(0,N)$
			states that for all $N'\geq N$ the functional ${\rm Ent}_{N'}(\cdot|\,u)$ is
			convex on the $L^2$-Wasserstein space 	$\mathcal P(X,{\sf d}, {\sf m})$, i.e.,  for each
			$ {\sf m}_0, {\sf m}_1\in  	\mathcal P(X,{\sf d}, {\sf m})$ there exists
			a geodesic
			$\Gamma:[0,1]\to  	\mathcal P(X,{\sf d}, {\sf m})$ joining
			$ {\sf m}_0$ and $ {\sf m}_1$ such that for every $s\in [0,1]$ one has 
			$${\rm Ent}_{N'}(\Gamma(s)| {\sf m})\leq (1-s) {\rm Ent}_{N'}( {\sf m}_0| {\sf m})+s {\rm Ent}_{N'}( {\sf m}_1| {\sf m}).$$


			If $(X,{\sf d}, {\sf m})$ is a metric measure space satisfying the ${\sf CD}(0,N)$ condition for some $N>1$,  the Bishop-Gromov volume comparison principle states that 
			$$ r\mapsto \frac{ {\sf m}(B(x,r))}{r^{N}},\ \ r>0,$$
			is non-increasing on $[0,\infty)$ for every $x\in X$, see Sturm \cite{Sturm-2}, where $B(x,r)=\{y\in X: {\sf d}(x,y)<r\}$ is the ball with center $x\in X$ and radius $r>0.$ This monotonicity property implies that  the \textit{asymptotic volume ratio}
			$${\sf AVR}(X)=\lim_{r\to \infty}\frac{ {\sf m}(B(x,r))}{\omega_Nr^N},$$
			is well-defined, i.e., it is independent of the choice of $x\in X$.

			We say that a  metric measure space $(X,{\sf d}, {\sf m})$  satisfies the \textit{Riemannian curvature-dimension condition} ${\sf RCD}(0,N)$ for $N>1$ if it is a ${\sf CD}(0,N)$ space and it is infinitesimally Hilbertian, i.e., the Banach space $W^{1,2}(X, {\sf d},{\sf m})$ is Hilbertian (or, equivalently, the Cheeger energy ${\sf Ch}$ is a quadratic form). Typical examples of ${\sf RCD}(0,N)$ spaces  include measured Gromov-Hausdorff limit spaces of Riemannian manifolds with non-negative Ricci curvature. Having an ${\sf RCD}(0,N)$ space $(X,{\sf d}, {\sf m})$, we consider the symmetric and bilinear form defined by
			$$\langle \nabla u_,\nabla w \rangle=\lim_{t\to 0}\frac{|\nabla(u+tw)|^2-|\nabla u|^2}{2t}\in L^1(X,{\sf m}),\ \ \forall u,w\in W^{1,2}(X, {\sf d},{\sf m})$$
			in ${\sf m}$-a.e.\ sense, where recall that $|\nabla \cdot|$ is the 
			the minimal relaxed slope. Having this object, we may define
			the linear Laplacian as follows: let $D(\Delta)$ be the set of all functions  $u\in W^{1,2}(X, {\sf d},{\sf m})$ such that there exists $h\in L^2(X,{\sf m})$ with the property that 
			\begin{equation}\label{divergence}
				\int_Xh  \psi\ {\rm d} {\sf m}=-\int_X\langle\nabla u,\nabla \psi\rangle\ {\rm d} {\sf m}, \ \ \forall \psi\in W^{1,2}(X,{\sf d}, {\sf m}).
			\end{equation}
			Since such an $h$ is unique whenever it exists, this function will be the Laplacian of $u$, denoted by   $\Delta u$. The infinitesimally Hilbertian structure of $(X,{\sf d}, {\sf m})$ implies that $u\mapsto \Delta u$ is linear; the set  $D(\Delta)$ is the domain of $\Delta$.
			
			Finally, let us recall a special subclass of ${\sf RCD}(0,N)$ spaces, the so-called \textit{non-collapsed spaces} introduced by De Philippis and Gigli \cite{DG}; an ${\sf RCD}(0,N)$ space $(X, {\sf d}, {\sf m})$ is said to be non-collapsed if the reference measure ${\sf m}$ coincides with the Hausdorff measure $\mathcal{H}^N$ of dimension $N$, ${\sf m}=\mathcal{H}^N$. Its fundamental properties include that $N$ must be an integer, that ${\sf AVR}(X) \le 1$ holds, and that  ${\sf AVR}(X) = 1$ holds if and only if $X$ is $\mathbb{R}^N$; moreover, if ${\sf AVR}(X)$ is close to $1$, then $X$ is homeomorphic to $\mathbb{R}^N$ due to the Reifenberg method by Cheeger and Colding \cite{CheegerColding1}  combined with Kapovitch and Mondino \cite{KM}; see also Huang and Huang \cite{HH} for a refinement.
			
			It is worth mentioning that the above can be generalized to ${\sf CD}(K,N)$ and ${\sf RCD}(K,N)$ for general $K \in \mathbb{R}$, though they are outside of our scopes.
			
			\subsection{Heat flows} 
			
			Let  $(X,{\sf d}, {\sf m})$ be an infinitesimally Hilbertian metric measure space. Given $f\in L^2(X,{\sf m})$, there exists a unique locally absolutely continuous curve $t\mapsto {\sf H}_t f\in L^2(X,{\sf m})$, $t>0$, called the \textit{heat flow} with initial datum $f$, such that  ${\sf H}_t f\to f$ in $L^2(X,{\sf m})$ whenever $t\to 0$, ${\sf H}_t f\in D(\Delta)$ for any $t>0$, and verifying the  equation 
			\begin{equation}\label{heat-equation}
				\frac{\rm d}{{\rm d}t}{\sf H}_t f=\Delta {\sf H}_t f, \quad \text{for $\mathcal{L}^1$-a.e. $t>0$,}
			\end{equation}
			see Ambrosio,  Gigli, and Savar\'e \cite{AGS} and Gigli \cite{gigli}. More strongly, the heat flow is actually smooth, and it also acts on $L^p(X, {\sf m})$ together with the contractivity $\|{\sf H}_tf\|_{L^p} \le \|f\|_{L^p}$ for any $1 \le p \le \infty$, see Gigli and Pasqualetto \cite{GP} for more details.
			
			We notice that for the heat semigroup $({\sf H}_t)_{t\geq 0}$ the domain $D(\Delta)$  of the Laplacian $\Delta$ is the space of functions 
			$f\in  L^2(X,{\sf m})$ such that the limit 
			\begin{equation}\label{limit-in-zero}
				\lim_{t\to 0}\frac{{\sf H}_t f -f}{t} 
			\end{equation}
			exists in $L^2(X,{\sf m})$; this limit is nothing but  $ \Delta f$, which is the infinitesimal Markov generator $\Delta$. The  domain $D(\Delta)$  of the Laplacian is a dense subset of $L^2(X,{\sf m}).$ 
			
			Let us prove the following injectivity of the heat flow. This will play a key role in showing the uniqueness of the optimizer in Theorem \ref{theorem-equality-RCD}.
			\begin{proposition}[Injectivity of the heat flow]\label{proposition-unique}
				Let $(X,{\sf d},{\sf m})$ be an infinitesimally Hilbertian metric measure  space.
				If ${\sf H}_{t} f=0$ for some $t>0$, some $f \in L^p\cap L^{\frac{p}{p-1}}(X,{\sf m})$, 
				and some $1 \le p \le \infty$, then $f=0$. 
			\end{proposition}
			\begin{proof}
				Noticing 
				\begin{equation}\label{equation lp-l2}
					\|{\sf H}_sf\|_{L^2(X,{\sf m})}^2\le \|{\sf H}_sf\|_{L^p(X,{\sf m})} \cdot \|{\sf H}_sf\|_{L^{\frac{p}{p-1}}(X,{\sf m})} \le \|f\|_{L^p(X,{\sf m})} \cdot \|f\|_{L^{\frac{p}{p-1}}(X,{\sf m})}<\infty,
				\end{equation}
				put
				$$
				E(s)=\|{\sf H}_sf\|_{L^2(X,{\sf m})}^2.
				$$
				Recall that
				$$
				\|\Delta {\sf H}_ug\|_{L^2(X,{\sf m})}^2 \le \frac{\|g\|_{L^2(X,{\sf m})}^2}{u^2},\quad \forall g \in L^2(X, {\sf m}),\quad \forall u>0,
				$$
				thus this is valid as $g={\sf H}_sf$  because of (\ref{equation lp-l2}). 
				Then,  by \eqref{heat-equation}, one  has
				$$
				-E'(s)=-2\int_X{\sf H}_sf \cdot \Delta {\sf H}_sf {\rm d}{\sf m} \le 2\|{\sf H}_s f\|_{L^2(X,{\sf m})} \cdot \|\Delta {\sf H}_sf\|_{L^2(X,{\sf m})},
				$$
				with $-E'(s)=-2\ds\int_X{\sf H}_sf \cdot \Delta {\sf H}_sf {\rm d}{\sf m} =2\int_X\nabla {\sf H}_sf \cdot\nabla {\sf H}_sf {\rm d}{\sf m}\geq 0$,
				and 
				$$
				E''(s)=2\int_X\Delta {\sf H}_sf \cdot \Delta {\sf H}_sf {\rm d}{\sf m} + 2\int_X{\sf H}_sf \cdot \Delta (\Delta {\sf H}_sf) {\rm d}{\sf m} =4 \|\Delta {\sf H}_sf\|_{L^2(X,{\sf m})}^2,
				$$
				thus obtaining 
				\begin{equation}\label{eq:convex}
					(E'(s))^2 \le E(s) \cdot E''(s).
				\end{equation}
				Define
				$$
				t_0=\inf \{s \in (0, t] | E(s)=0\}, 
				$$
				where $t_0 \le t$ from our assumption.
				
				Let us prove $t_0=0$ by contradiction. Assume  $t_0>0$. For any $0<t_1<t_0$, put
				$$
				g(s)=\log E(s), \quad  s\in [t_1, t_0).
				$$
				Then, by (\ref{eq:convex}),
				$$
				g''(s)=\frac{E''(s)}{E(s)}-\frac{E'(s)^2}{E(s)^2} \ge 0,
				$$
				which implies that $g$ is convex on $[t_1, t_0)$; in particular,
				$$
				g((1-\theta)t_1+\theta t) \le (1-\theta)g(t_1)+\theta g(t),\quad \text{for all $t_1 \le t<t_0$ and $0<\theta <1$.}
				$$
				In other words,
				\begin{equation}\label{eq:convex energy}
					0\le E((1-\theta)t_1+\theta t) \le E(t_1)^{1-\theta} \cdot E(t)^{\theta}.
				\end{equation}
				Since letting $t \to t_0^-$ shows that the right hand side of (\ref{eq:convex energy}) converges to $0$, we have $E(s)=0$ on $[t_1, t_0),$ which contradicts the definition of $t_0$.
				
				Therefore we can find a sequence $s_i \to 0^+$ with ${\sf H}_{s_i}f=0$. Since ${\sf H}_sf \to f$ in $L^{\min \{p, \frac{p}{p-1}\}}(X, {\sf m})$ as $s \to 0^+$, we conclude that $f=0$.
			\end{proof}
			We now prove a basic property of the heat flow.
			\begin{proposition}\label{proposition-test}
				Let $(X,{\sf d},{\sf m})$ be an  ${\sf RCD}(K,\infty)$ space for some $K \in \mathbb{R}$, and let $f \in L^1 \cap L^{\infty}(X, {\sf m})$. Then ${\sf H}_sf \in \mathrm{Test}F(X, {\sf d}, {\sf m})$ for any $s>0$, where 
				$$
				\mathrm{Test}F(X, {\sf d}, {\sf m})=\left\{ f \in D(\Delta) \cap \mathrm{Lip} (X) \cap L^{\infty}(X, {\sf m})| \Delta f \in W^{1,2}(X, {\sf d}, {\sf m})\right\}.
				$$
				In particular, for all non-negative $f\in L^1(X,{\sf m})\cap L^\infty(X,{\sf m})$, $t>0$, and $\alpha\geq 1$, one has that $({\sf H}_t f)^\alpha\in W^{1,2}(X,{\sf d},{\sf m})$. 
			\end{proposition}
			\begin{proof}
				As already checked in (\ref{equation lp-l2}), we  have ${\sf H}_sf \in L^2 \cap L^{\infty}(X, {\sf m})$ for any $s>0$. Thus, considering the identity ${\sf H}_sf={\sf H}_{\frac{s}{2}} \circ {\sf H}_{\frac{s}{2}} f$, we obtain the first statement because of a general fact for ${\sf RCD}(K, \infty)$: 
				if $f \in L^2 \cap L^{\infty}(X, {\sf m})$, then ${\sf H}_sf \in \mathrm{Test}F(X, {\sf d}, {\sf m})$ for any $s>0$.
				The last statement is a direct consequence of this  together with a general fact (with no use of {\sf RCD} properties): if $g \in W^{1,2}(X, {\sf d}, {\sf m}) \cap L^{\infty}(X, {\sf m})$, then $|g|^{\alpha} \in W^{1,2}(X, {\sf d}, {\sf m})$ for any $\alpha \ge 1$.
			\end{proof}
			Finally, let us mention  a few facts about the heat kernel. In the case when $(X, {\sf d}, {\sf m})$ is an ${\sf RCD}(K, N)$ space for some $K \in \mathbb{R}$ and some $N \in [1, \infty)$, the heat flow can be written as
			\begin{equation}\label{heat-integral-representation}
				{\sf H}_tf(x)=\int_Xf(y){\sf h}(x, y, t){\rm d}{\sf m}(y), \quad \forall f \in L^2(X, {\sf m}),
			\end{equation}
			for  a (uniquely determined) locally Lipschitz symmetric function ${\sf h}:X \times X \times (0, \infty) \to (0, \infty)$,  which also satisfies the stochastic completeness property $\|{\sf h}(x, \cdot, t)\|_{L^1}=1$,  called the \textit{heat kernel} of $(X, {\sf d}, {\sf m})$. 
			Moreover, a Gaussian estimate for ${\sf h}$ (and also for its gradient) is known due to  Jiang, Li, and Zhang \cite{JLZ}. Let us write it down in the case when $K=0$: for any $0<\epsilon\leq 1$, there exists  $C=C(N, \epsilon)>1$ such that for every $x,y\in X$, 
			\begin{equation}\label{equation-gaussian}
				\frac{1}{C{\sf m}(B_{\sqrt{t}}(x))}e^{-\frac{{\sf d}(x,y)^2}{(4-\epsilon)t}}\le {\sf h}(x, y, t) \le \frac{C}{{\sf m}(B_{\sqrt{t}}(x))}e^{-\frac{{\sf d}(x,y)^2}{(4+\epsilon)t}}.
			\end{equation}
			
			Additionaly, note that if $(X, {\sf d}, {\sf m})$ is an $N$-Euclidean cone over an ${\sf RCD}(N-2, N-1)$ space with a tip  $x_0 \in X$ and ${\sf AVR}(X)>0$, then the heat kernel of $(X, {\sf d}, {\sf m})$  is given by 
			\begin{equation}\label{heat-kernel-on-cone}
				{\sf h}(x_0, x, s)=\frac{{\sf AVR}(X)^{-1}}{(4\pi s)^{N/2}}e^{-\frac{{\sf d}^2(x_0,x)}{4s}},
			\end{equation}
			see  Honda and Peng \cite{Honda-Peng}.

			\subsection{Sharp logarithmic-Sobolev inequality on  ${\sf RCD}(0,N)$ spaces} \label{section-log-sob}
			In the sequel, let $(X,{\sf d},{\sf m})$ be an   ${\sf RCD}(0,N)$ metric measure space with $N>1$ and ${\sf AVR}(X) >0$. According to Balogh, Krist\'aly, and Tripaldi  \cite{BKT}, one has the  
			$L^2$-logarithmic-Sobolev inequality
			$$\int_X u^2 \log u^2 \, {\rm d}{\sf m} \le \frac{N}{2}\log\left(\frac{2}{N\pi e}{\sf AVR}(X)^{-\frac{2}{N}}\int_X \abs{\nabla u}^2 \, {\rm d}{\sf m}\right),\ \ \  \forall u\in W^{1,2}(X,{\sf d},{\sf m}), \int_X u^2 \, {\rm d}{\sf m}=1,$$
			and the constant $\frac{2}{N\pi e}{\sf AVR}(X)^{-\frac{2}{N}}$ is sharp. In the non-normalized form, the latter inequality reads as 
			\begin{equation}\label{log-Sobolev}
				\frac{{\mathcal E_{\sf m}}(u^2)}{\|u\|^2_{L^2(X,{\sf m})}}\le \frac{N}{2}\log\left(\frac{2}{N\pi e}{\sf AVR}(X)^{-\frac{2}{N}}\frac{\ds\int_X \abs{\nabla u}^2 \, {\rm d}{\sf m}}{\|u\|^2_{L^2(X,{\sf m})}}\right),\ \ \  \forall u\in W^{1,2}(X,{\sf d},{\sf m})\setminus \{0\},
			\end{equation}
			where 
			$$\mathcal E_{\sf m}(u)=\int_X u\log u\, {\rm d}{\sf m} -\left(\int_X u\, {\rm d}{\sf m}\right)\log \left(\int_X u\, {\rm d}{\sf m}\right),\ u\in L^1(X,{\sf m}),\ u\geq 0,$$
			is the entropy of $u$ with respect to the measure ${\sf m}$. Due to Nobili and Violo \cite{Nobili-Violo} (see also Krist\'aly \cite{Kristaly-new}), equality holds in \eqref{log-Sobolev} for some $u_0\in W^{1,2}(X,{\sf d},{\sf m})\setminus \{0\}$ if and only if $(X,{\sf d},{\sf m})$ is an $N$-\textit{Euclidean cone} with a tip $x_0\in X$ and $u$ is a Gaussian-type function of the form $$u_0(x)=c_1 e^{-c_0{\sf d}^2(x_0,x)}\ \ {\rm for}\  {\sf m}-{\rm a.e}\ x\in X,$$  where $c_1\in \mathbb R$, $c_0>0.$ Here,  $(X,{\sf d},{\sf m})$ is an $N$-Euclidean cone over a compact ${\sf RCD}(N-2,N-1)$ metric measure space $(Z,{\sf d}_Z,{\sf m}_Z)$ if  $(X,{\sf d}, {\sf m})$ is isometric to the metric measure cone $(C(Z),{\sf d}_c,t^{N-1}{\rm d}t\otimes {\sf m}_Z)$, where $C(Z)=Z \times [0,\infty)/(Z \times \{0\})$ and ${\sf d}_c$ is the usual cone distance from ${\sf d}_Z$.
			

			\section{Proof of Theorem \ref{main-theorem-RCD}} \label{section-3}
			
			\subsection{(i)$\implies$(ii)  of Theorem \ref{main-theorem-RCD}: positivity of ${\sf AVR}$ from asymptotic   hypercontractivity} Let $(X, {\sf d}, {\sf m})$ be an ${\sf RCD}(0,N)$ space for some $N \ge 1$. In this subsection, we are going to prove that if 
			for some $1 \le p<q \le \infty$ one has
			\begin{equation}\label{assumption-infinity}
				\liminf_{t \to \infty}(4\pi t)^{\frac{N}{2}(\frac{1}{p}-\frac{1}{q})}\|{\sf H}_t\|_{p, q}<\infty,
			\end{equation}
			then 
			$
			{\sf AVR}(X)>0.
			$ The following two cases will be discussed. 
			
			\textbf{Case 1:}  $q=\infty$.
				%
				For a fixed $x_0 \in X$, let us consider $f(x)={\sf h}(x_0, x, t)$. Then, by the Gaussian estimate (\ref{equation-gaussian}), there exists $C(N,p)>0$ such that 
				\begin{align}\label{align-lp}
					\|f\|_{L^p(X,{\sf m})} &\le \frac{C(N,p)}{{\sf m}(B_{\sqrt{t}}(x_0))} \cdot \left(\int_Xe^{-\frac{p{\sf d}(x_0, x)^2}{5t}}{\rm d}{\sf m}\right)^{\frac{1}{p}} \nonumber \\
					&\le \frac{C(N,p)}{{\sf m}(B_{\sqrt{t}}(x_0))} \cdot {\sf m}(B_{\sqrt{t}}(x_0))^{\frac{1}{p}}=C(N,p){\sf m}(B_{\sqrt{t}}(x_0))^{\frac{1}{p}-1},
				\end{align}
				where we used  the Cavalieri and the Bishop--Gromov principles, see e.g.\ Brena, Gigli, Honda, and Zhu \cite{BGHZ}. 
				On the other hand, 
				$$
				{\sf m}(B_{\sqrt{t}}( x_0))  {\sf H}_tf( x_0)={\sf m}(B_{\sqrt{t}}(x))  {\sf h}(x_0,  x_0, 2t)
				$$
				has a uniform positive lower bound
				due to the  Gaussian estimate (\ref{equation-gaussian}) again. 
				Thus, by \eqref{align-lp}, one has
				\begin{align}\label{align-vol}
					\left(\frac{\sqrt{t}^N}{{\sf m}(B_{\sqrt{t}}(x_0))}\right)^{\frac{1}{p}} &\le C \left(\frac{\sqrt{t}^N}{{\sf m}(B_{\sqrt{t}}(x_0))}\right)^{\frac{1}{p}} \cdot {\sf m}(B_{\sqrt{t}}(x_0)) \cdot {\sf h}(x_0, x_0, 2t) \nonumber \\
					&\le C \cdot \sqrt{t}^{\frac{N}{p}} \cdot \frac{1}{\|f\|_{L^p(X,{\sf m})}} \cdot \|{\sf H}_tf\|_{L^{\infty}(X,{\sf m})} \le C \cdot t^{\frac{N}{2p}}\|{\sf H}_t\|_{p, \infty}.
				\end{align}
				Then, letting $t \to \infty$ and using \eqref{assumption-infinity}, we obtain  $
				{\sf AVR}(X)>0.
				$
				
				\textbf{Case 2:}  $q<\infty$.  We consider the same function $f(x)={\sf h}(x_0, x, t)$ as above.  By using again the  Gaussian estimate (\ref{equation-gaussian}) and  the Bishop--Gromov principle,  for every $x\in B_{\sqrt{t}}(x_0)$, one has 
				\begin{eqnarray*}
					{\sf m}(B_{\sqrt{t}}( x_0))  {\sf H}_tf(x)&=&{\sf m}(B_{\sqrt{t}}(x_0))  {\sf h}(x_0, x, 2t)=\frac{{\sf m}(B_{\sqrt{t}}(x_0))}{{\sf m}(B_{\sqrt{2t}}(x_0))}{\sf m}(B_{\sqrt{2t}}(x_0))  {\sf h}(x_0, x, 2t)\\&\geq& \frac{1}{ C 2^\frac{N}{2}}e^{-\frac{1}{6}}. 
				\end{eqnarray*}
				where $C=C(N,1)>0$ is from (\ref{equation-gaussian}).


				Considering the rescaling
				$$
				\left(X, \frac{1}{\sqrt{t}}{\sf d}, \frac{1}{{\sf m}(B_{\sqrt{t}}(x_0))}{\sf m}\right),
				$$ 
				we see by the above estimate that 
				$$
				\frac{1}{{\sf m}(B_{\sqrt{t}}(x_0))}\int_{B_{\sqrt{t}}(x_0)}\left({\sf m}(B_{\sqrt{t}}(x_0)){\sf h}(x_0, x, 2t)\right)^q{\rm d}{\sf m}
				$$
				has a uniform positive lower  bound as $t \to \infty$.
				Then, by \eqref{align-lp} and the latter fact, one has that
				\begin{align*}
					\frac{\|{\sf H}_tf\|_{L^q(X,{\sf m})}}{\|f\|_{L^p(X,{\sf m})}}& \ge \frac{1}{C(N,p){\sf m}(B_{\sqrt{t}}(x_0))^{\frac{1}{p}-1}}\cdot \left( \int_{B_{\sqrt{t}}(x_0)}{\sf h}(x_0, x, 2t)^q{\rm d}{\sf m}\right)^{\frac{1}{q}} \nonumber \\
					&= \frac{{\sf m}(B_{\sqrt{t}}(x_0))^{\frac{1}{q}-1}}{C(N,p){\sf m}(B_{\sqrt{t}}(x_0))^{\frac{1}{p}-1}}\cdot \left(\frac{1}{{\sf m}(B_{\sqrt{t}}(x_0))}\int_{B_{\sqrt{t}}(x_0)}\left({\sf m}(B_{\sqrt{t}}(x_0)){\sf h}(x_0, x, 2t)\right)^q{\rm d}{\sf m}\right)^{\frac{1}{q}} \nonumber \\
					&\ge C{\sf m}(B_{\sqrt{t}}(x_0))^{\frac{1}{q}-\frac{1}{p}}.
				\end{align*}
				Thus, multiplying the latter estimate by $t^{\frac{N}{2}\left(\frac{1}{p}-\frac{1}{q}\right)}$ and then letting $t \to \infty$, the assumption \eqref{assumption-infinity} together with $ p<q$  implies again that $
				{\sf AVR}(X)>0.
				$

			\subsection{(ii)$\implies$(i)  of Theorem \ref{main-theorem-RCD}: proof of the hypercontractivity bound \eqref{Heat-bound-main}} Let $(X, {\sf d}, {\sf m})$ be an ${\sf RCD}(0,N)$ space for some $N \ge 1$ with ${\sf AVR}(X)>0.$
			The proof of \eqref{Heat-bound-main} can be obtained directly from Bakry, Gentil, and Ledoux \cite[Theorem 7.1.2]{BGL}; however, in order to describe  the equality case in \eqref{Heat-bound-main}, we recall the main steps of the proof in our setting. 
			
			Let us first treat the case when $2\leq p\leq q\leq  \infty;$ in fact, we shall consider the case $2\leq p< q< \infty,$ the limit cases $p=q$ and $q=\infty$ being obtained  after a limiting argument.  
			Let $\Phi:(0,\infty)\to \mathbb R$ be the function  given by $$\Phi(s)=\frac{N}{2}\log\left(\frac{2}{N\pi e}{\sf AVR}(X)^{-\frac{2}{N}}s\right),\ \ s>0.$$
			With this notation, \eqref{log-Sobolev}	 reads equivalently as 
			\begin{equation}\label{log-Sobolev-second}
				{\mathcal E_{\sf m}}(u^2)\le \|u\|^2_{L^2(X,{\sf m})}\Phi\left(\frac{\ds\int_X \abs{\nabla u}^2 \, {\rm d}{\sf m}}{\|u\|^2_{L^2(X,{\sf m})}}\right),\ \ \  \forall u\in W^{1,2}(X,{\sf d},{\sf m})\setminus \{0\},
			\end{equation}
			Since $\Phi$ is a concave function of class $C^1$, we have 
			\begin{equation}\label{concave}
				\Phi(s)\le \Phi(v)+\Phi'(v)(s-v),\ \ \ \forall s, v>0.
			\end{equation}
			Combining \eqref{log-Sobolev-second} with the latter relation for the choice $$s:=\frac{{\ds\int_X \abs{\nabla u}^2 \, {\rm d}{\sf m}}}{{\|u\|^2_{L^2(X,{\sf m})}}},$$ with $ u\in W^{1,2}(X,{\sf d},{\sf m})\setminus \{0\}$, it turns out that for every   $v>0$ one has
			\begin{equation}\label{log-Sobolev-third}
				{\mathcal E_{\sf m}}(u^2)\le \|u\|^2_{L^2(X,{\sf m})}\left(\Phi(v)-\frac{N}{2}\right) + \frac{N}{2v}\int_X \abs{\nabla u}^2 \, {\rm d}{\sf m}.
			\end{equation}
			
			Now let us fix a non-negative function $f\in L^1(X,{\sf m})\cap L^\infty(X,{\sf m})$ (the proof of \eqref{Heat-bound-main} for an arbitrary function $f\in L^1(X,{\sf m})\cap L^\infty(X,{\sf m})$ will follow by a straightforward density argument). Consider the heat flow with initial datum $f$, i.e.,  $$f_t:={\sf H}_t f,\ \ t>0.$$
			
			For some $T:=T_\lambda>0$ -- which will depend  on a parameter $\lambda>0$ with $T_\lambda\to \infty$ as $\lambda\to 0$ and   will be determined later --   let us consider a continuously differentiable, increasing function $[0, T)\ni t \mapsto p(t)\in [2, \infty)$ such that $p(0)=p$. Let \mbox{$m:[0, T)\to \RR$} be another  continuously differentiable function  such that $m(0)=0$. 
			We define the function $V:[0,T)\to \mathbb R$ by $$V(t)=e^{-m(t)}\norm{f_t}_{L^{p(t)}(X,{\sf m})}.$$
			Our general goal is to show that for certain choices of  $p$ and $m$, the function $V$ is \mbox{nonincreasing} on the interval $[0,T)$.

			According to Proposition \ref{proposition-test}, since  $p(t)\geq 2$ for $t\in [0,T),$ it turns out that 
			$$u:=f_t^{p(t)/2}\in W^{1,2}(X,{\sf d},{\sf m}).$$
			In particular, this function can be used in  \eqref{log-Sobolev-third}, obtaining that
			\begin{equation}\label{log-Sobolev-4}
				{\mathcal E_{\sf m}}\left(f_t^{p(t)}\right)\le \|f_t\|^{p(t)}_{L^{p(t)}(X,{\sf m})}\left(\Phi(v)-\frac{N}{2}\right) + \frac{N p^2(t)}{8v}\int_X f_t^{p(t)-2}\abs{\nabla f_t}^2 \, {\rm d}{\sf m},\ \ \forall v>0,\ t\in [0,T).
			\end{equation}
			By using the heat equation  \eqref{heat-equation} and relation \eqref{divergence}, one has that 
			\begin{align*}
				\frac{\rm d}{{\rm d}t}\|f_t\|^{p(t)}_{L^{p(t)}(X,{\sf m})}=&\frac{\rm d}{{\rm d}t}\int_X f_t^{p(t)}{\rm d}{\sf m}=\int_X f_t^{p(t)-1}\left(p'(t)f_t\log f_t+p(t)f_t'\right){\rm d}{\sf m}\\=&\int_X f_t^{p(t)-1}\left(p'(t)f_t\log f_t+p(t)\Delta f_t\right){\rm d}{\sf m}\\=&
				\frac{	p'(t)}{p(t)}\int_X f_t^{p(t)}\log f_t^{p(t)}{\rm d}{\sf m} -p(t)(p(t)-1)\int_X f_t^{p(t)-2}\abs{\nabla f_t}^2 \, {\rm d}{\sf m}.
			\end{align*}
			Therefore, the latter relation and the  chain rule give 
			\begin{equation}\label{derivative-1}
				\frac{\rm d}{{\rm d}t}\|f_t\|_{L^{p(t)}(X,{\sf m})}=\frac{p'(t)}{p^2(t)}\norm{f_t}_{L^{p(t)}(X,{\sf m})}^{1-p(t)}\left(	{\mathcal E_{\sf m}}\left(f_t^{p(t)}\right)-\frac{p^2(t)(p(t)-1)}{p'(t)}\int_X f_t^{p(t)-2}\abs{\nabla f_t}^2 \, {\rm d}{\sf m}\right).
			\end{equation}
			Using \eqref{derivative-1} together with \eqref{log-Sobolev-4} for $v:=v(p(t))$, where $v:[2,\infty)\to (0,\infty)$ will  also be determined later, we obtain 
			\begin{align}\label{V-deirvalt}
				V'(t)=&e^{-m(t)}\left(-m'(t)\norm{f_t}_{{L^{p(t)}(X,{\sf m})}}+\frac{\rm d}{{\rm d}t}\norm{f_t}_{L^{p(t)}(X,{\sf m})}\right)\\ \nonumber \leq &e^{-m(t)}\left\{\left(-m'(t)+\frac{p'(t)}{p^2(t)}\left(\Phi(v(p(t)))-\frac{N}{2}\right)\right)\norm{f_t}_{L^{p(t)}(X,{\sf m})}
				\right.\\ &\nonumber \ \ \ \ \ \ \ \ \ \ \ \ \ \ + \left.{p'(t)}\left(\frac{N}{8v(p(t))}-\frac{p(t)-1}{p'(t)}\right)\norm{f_t}_{L^{p(t)}(X,{\sf m})}^{1-p(t)}\int_X f_t^{p(t)-2}\abs{\nabla f_t}^2 \, {\rm d}{\sf m}\right\}.
			\end{align}

					If we choose the functions $p,m$, and $v$ such that \begin{equation}\label{integrate}
						m'(t)=\frac{p'(t)}{p^2(t)}\left(\Phi(v(p(t)))-\frac{N}{2}\right), \quad	\frac{N}{8v(p(t))}=\frac{p(t)-1}{p'(t)}  \quad \forall t\in [0,T),
					\end{equation}
					then  $V'(t)\le 0$ for every $t\in [0,T),$ i.e.\ $V$ is nonincreasing on $[0,T)$. In particular, under \eqref{integrate}, it  follows that
					$V(t)\le \lim_{s\to 0^{+}}V(s)$ for every $t\in [0,T),$ i.e. 
					\begin{equation}\label{estim}
						\norm{{\sf H}_{t}f}_{L^{p(t)}(X,{\sf m})}\le e^{m(t)}\norm{f}_{L^{p}(X,{\sf m})},\ \  \forall t\in [0,T).
					\end{equation}
					
					In the sequel, we focus on the ODEs in \eqref{integrate}. The second equation from \eqref{integrate}  is equivalent to $$1= \frac{Np'(t)}{8(p(t)-1)v(p(t))},\ \  \forall t\in [0,T).$$ An integration from $0$ to $t$ and a  change of variables  yield  
					\begin{equation}\label{tpq}
						t=\frac{N}{8}\int_p^{p(t)} \frac{{\rm d}s}{(s-1)v(s)},\ \  \forall t\in [0,T).
					\end{equation}
					In the same way, from the first equation of \eqref{integrate} we obtain
					\begin{equation}\label{mpq}
						m(t)=\int_p^{p(t)}\left(\Phi(v(s))-\frac{N}{2}\right)\frac{{\rm d}s}{s^2}.
					\end{equation}
					According to Bakry \cite{Bakry}, the optimal choice of the function $v$  is 
					\begin{equation}\label{lambdav}
						v(s)=\lambda \frac{s^2}{s-1},\ \  s\geq 2,
					\end{equation} where $\lambda>0$ is a parameter (that we have already mentioned in the definition of $T:=T_\lambda$).

					Coming back to  \eqref{tpq}, an elementary computation gives us 
					\begin{equation}\label{tpqlambda}
						t=\frac{N}{8\lambda}\left(\frac{1}{p}-\frac{1}{p(t)}\right), \ \ t\in [0,T).
					\end{equation}
					If $\lambda>0$ is fixed, by \eqref{tpqlambda} it is clear that the admissible range for $t$ is the interval $[0,T)$, where 
					$$T:=T_\lambda=\frac{N}{8\lambda p}.$$
					Now, \eqref{tpqlambda} can be transformed into 
					$p(t)=\frac{Np}{N-8\lambda p t},\ \  \forall t\in [0,T_\lambda).$
					
					By \eqref{mpq} and \eqref{lambdav}, we also obtain  
					\begin{equation}\label{mpqlambda1}
						m(t)=\frac{N}{2}\left[\log\left(\frac{2\lambda}{N\pi }\operatorname{\sf AVR}(X)^{-\frac{2}{N}}\right)\left(\frac{1}{p}-\frac{1}{p(t)}\right)+\log \frac{p^{\frac{1}{p}}\left(1-\frac{1}{p}\right)^{1-\frac{1}{p}}}{p(t)^{\frac{1}{p(t)}}\left(1-\frac{1}{p(t)}\right)^{1-\frac{1}{p(t)}}}\right].
					\end{equation}
					Let us observe that $t\mapsto p(t)$, $t\in [0,T_\lambda),$ is indeed an increasing continuous function and its range is the interval $[p, \infty)$ as $p(0)=p $ and $\lim_{t\to T_\lambda^{-}}p(t)=\infty.$
					Consequently,  for every $\lambda >0$ there is a point $t(\lambda)\in (0, T_\lambda)$ such that $p(t(\lambda))=q$.
					
					Therefore, if we take $t:=t(\lambda)$ in \eqref{tpqlambda}, we have 
					\begin{equation}\label{lambdatlambda}
						\lambda=\frac{n}{8t(\lambda)}\left(\frac{1}{p}-\frac{1}{q}\right),
					\end{equation}
					so if we substitute this expression for $\lambda$ into \eqref{mpqlambda1}, we obtain  
					$$e^{m(t(\lambda))}=\frac{{\sf M}(p,q)^\frac{N}{2}{\sf AVR}(X)^{\frac{1}{q}-\frac{1}{p}}}{(4\pi t(\lambda))^{\frac{N}{2}(\frac{1}{p}-\frac{1}{q})}},$$
					where 
					$${\sf M}(p,q)=\frac{p^{\frac{1}{p}}\left(1-\frac{1}{p}\right)^{1-\frac{1}{p}}}{q^{\frac{1}{q}}\left(1-\frac{1}{q}\right)^{1-\frac{1}{q}}}\left(\frac{1}{p}-\frac{1}{q}\right)^{\frac{1}{p}-\frac{1}{q}}.$$

					Now, relation \eqref{estim} for $t=t(\lambda)$ reads as 
					\begin{equation}\label{estimlam}
						\norm{{\sf H}_{t(\lambda)}f}_{L^{q}(X,{\sf m})}\le e^{m(t(\lambda))}\norm{f}_{L^{p}(X,{\sf m})}.
					\end{equation}
					Finally, due to \eqref{lambdatlambda}, one has that $t(\lambda)=\frac{N}{8\lambda}\left(\frac{1}{p}-\frac{1}{q}\right)$, thus   $\lim_{\lambda\to 0^{+}}t(\lambda)=\infty$ and $ \lim_{\lambda\to \infty}t(\lambda)=0,$ which shows that the continuous mapping $(0, \infty)\ni \lambda \mapsto t(\lambda)\in (0, \infty)$ is surjective. In this way,  we can replace $t(\lambda)$ in \eqref{estimlam}  by any $t>0$, obtaining 
					\begin{equation}\label{estimpq}
						\norm{{\sf H}_t f}_{L^q(X,{\sf m})}\le \frac{{\sf M}(p,q)^\frac{N}{2}{\sf AVR}(X)^{\frac{1}{q}-\frac{1}{p}}}{(4\pi t)^{\frac{N}{2}(\frac{1}{p}-\frac{1}{q})}}\norm{f}_{L^p(X,{\sf m})}, \  \forall t>0,
					\end{equation} 
					which is the required inequality \eqref{Heat-bound-main} for $2\leq p\leq q\leq  \infty.$ 
					
					If $1\leq  p\leq q\leq  2,$ we consider the case $1<  p< q\leq  2,$  the borderline cases $p=1$ and $p=q$ being obtained by a limiting argument. Let $p'=\frac{p}{p-1}$ and $q'=\frac{q}{q-1}$ be  the conjugates of $p$ and $q$, respectively. The  
					symmetry of the heat semigroup $({\sf H}_t)_{t\geq 0}$ implies that    $\|{\sf H}_t\|_{p,q}=\|{\sf H}_t\|_{q',p'}$, where  $$\|{\sf H}_t\|_{p,q}=\sup_{\|f\|_{L^p(X,{\sf m})}\leq 1}\|{\sf H}_tf\|_{L^q(X,{\sf m})}$$ stands for the operator norm from $L^p(X,{\sf m})$ to $L^q(X,{\sf m})$. Therefore, this property together with the previous step applied for $2\leq q'<p'<\infty$ implies \eqref{Heat-bound-main} since ${\sf M}(q',p')={\sf M}(p,q)$. 
					
					Finally, if  $p\leq  2\leq q,$ by using the property $\|{\sf H}_t\circ {\sf H}_s\|_{p,q}\leq \|{\sf H}_t\|_{p,r}\|{\sf H}_s\|_{r,q}$ for every $t,s\geq 0$ and $p\leq r\leq q$, we apply the previous steps to the pairs $(p,2)$ and $(2,q)$, respectively.

					\subsection{Optimality: proof of (\ref{sharp operator norm limit})}
					In this section we discuss the behavior of the sharp hypercontractivity estimates (\ref{Heat-bound-main}) with respect to \textit{pointed measured Gromov-Hausdorff (pmGH) convergence} under the assumption that the readers are familiar with this topic, including the $L^p$-convergence of functions on varying spaces with variable exponents. As a corollary, we will prove the sharpness stated in (\ref{sharp operator norm limit}). 
					
					In order to do so, let us introduce the following notation by emphasizing also the space-dependence of the heat flow. Namely, if 
					$(X,{\sf d},{\sf m})$ is an  infinitesimally Hilbertian metric measure space, $t>0$,  $1 \le p \le \infty$, and  $1 \le q \le \infty$, let 
					\begin{equation}\label{equation-optimal constant}
						{\sf C}(X, {\sf d}, {\sf m}, p, q, t):=\|{\sf H}_t\|_{p,q} \in [0, \infty].
					\end{equation}
					Recall that (\ref{Heat-bound-main}) tells us that if $(X,{\sf d},{\sf m})$ is an  ${\sf RCD}(0,N)$ space with ${\sf AVR}(X)>0$, then 
					\begin{equation}\label{eq:optimal}
						{\sf C}(X, {\sf d}, {\sf m}, p, q, t) \le \frac{{\sf M}(p,q)^\frac{N}{2}{\sf AVR}(X)^{\frac{1}{q}-\frac{1}{p}}}{(4\pi t)^{\frac{N}{2}(\frac{1}{p}-\frac{1}{q})}} \quad \text{for all $1 \le p\le q \le \infty$.}
					\end{equation}

					Our first observation is   the existence of extremizers for the constant ${\sf C}$ on Euclidean cones, a fact which will be used later. 
					\begin{proposition}[Existence]\label{proposition-metric cone}
						Let $(X, {\sf d}, {\sf m})$ be the $N$-Euclidean cone over an ${\sf RCD}(N-2, N-1)$ space with the tip  $x_0 \in X$, let $1< p<q\le \infty$, and let $t>0$. Then the function $f \in L^p \cap L^{\frac{p}{p-1}}(X, {\sf m})$ having the form
						$$
						f(x)={\sf h}(x_0, x, at),
						$$
						with $a=q\frac{p-1}{q-p}>0$, verifies the optimality in \eqref{eq:optimal}, i.e., 
						\begin{equation}
							{\sf C}(X, {\sf d}, {\sf m}, p,q,t)=\frac{\|{\sf H}_tf\|_{L^q(X,{\sf m})}}{\|f\|_{L^p(X,{\sf m})}}=\frac{{\sf M}(p,q)^\frac{N}{2}{\sf AVR}(X)^{\frac{1}{q}-\frac{1}{p}}}{(4\pi t)^{\frac{N}{2}(\frac{1}{p}-\frac{1}{q})}}.
						\end{equation}
					\end{proposition}
					\begin{proof}
						By \eqref{heat-kernel-on-cone}, one has for every $s>0$ that 
						$$
						\|{\sf h}(x_0, \cdot, s)\|_{L^p(X,  {\sf m})}={\sf AVR}(X)^{\frac{1}{p}-1} \cdot (4\pi s)^{\frac{N}{2}\left( \frac{1}{p}-1\right)} \cdot p^{-\frac{N}{2p}}.
						$$
						Thus, considering
						$
						f(x)={\sf h}(x_0, x, at),
						$
						we have
						$$
						\frac{\|{\sf H}_tf\|_{L^q(X,{\sf m})}}{\|f\|_{L^p(X,{\sf m})}}=\frac{p^{\frac{N}{2p}}}{q^{\frac{N}{2q}}}\cdot {\sf AVR}(X)^{\frac{1}{q}-\frac{1}{p}} \cdot (4\pi t)^{\frac{N}{2}\left(\frac{1}{q}-\frac{1}{p}\right)} \cdot \left( \frac{(1+a)^{\frac{1}{q}-1}}{a^{\frac{1}{p}-1}}\right)^{\frac{N}{2}}.
						$$
						Therefore, it is enough to show that the equation  
						\begin{equation}\label{eq:eq}
							\frac{(1+a)^{\frac{1}{q}-1}}{a^{\frac{1}{p}-1}}=\frac{\left(1-\frac{1}{p}\right)^{1-\frac{1}{p}}}{\left(1-\frac{1}{q}\right)^{1-\frac{1}{q}}} \cdot \left(\frac{1}{p}-\frac{1}{q}\right)^{\frac{1}{p}-\frac{1}{q}}
						\end{equation}
						has a solution in $a>0$. If $ h(a):=\frac{(1+a)^{\frac{1}{q}-1}}{a^{\frac{1}{p}-1}}$, $a>0$, an elementary argument shows that  $h(a)\to 0$ when $a\to 0$ and $a\to \infty$ (since $1< p<q\le \infty$), while the unique maximal point of $h$ is $a_{\rm max}=q\frac{p-1}{q-p}$. In addition, it turns out that $h(a_{\rm max})$   is precisely the expression from the right hand side of the above equation; therefore, we can choose $a:=a_{\rm max}>0$. 	
					\end{proof}

					\begin{corollary}[Explicit value of ${\sf C}$ on Euclidean cones]\label{cor:explicit}
						Let $(X, {\sf d}, {\sf m})$ be the $N$-Euclidean cone over an ${\sf RCD}(N-2, N-1)$ space with the tip  $x_0 \in X$. Then, for all $1 \le p \le q \le \infty$ and $t>0$,
						\begin{equation}\label{eq:optimal value}
							{\sf C}(X, {\sf d}, {\sf m}, p, q, t) = \frac{{\sf M}(p,q)^\frac{N}{2}{\sf AVR}(X)^{\frac{1}{q}-\frac{1}{p}}}{(4\pi t)^{\frac{N}{2}(\frac{1}{p}-\frac{1}{q})}}.
						\end{equation}
					\end{corollary}
					\begin{proof}
						Due to Proposition \ref{proposition-metric cone}, it is enough to discuss the following three limit cases. 
						
						\textit{Case 1:} $1=p<q$. In view of (\ref{Heat-bound-main}), since the left hand side of (\ref{eq:eq}) is equal to $(1+a)^{\frac{1}{q}-1}$ and the right hand side equals $1$,  it is enough to check that
						$
						\lim_{a \to 0^+}(1+a)^{\frac{1}{q}-1}=1,
						$
						which is trivial.
						
						\textit{Case 2:} $1<p=q$. In a similar way, it is enough to observe that
						$
						\lim_{a \to \infty}\frac{(1+a)^{\frac{1}{p}-1}}{a^{\frac{1}{p}-1}}=1.
						$
						
						\textit{Case 3:} $1=p=q$. This case is trivial,  since (\ref{eq:eq}) is valid by the stochastic completeness.
					\end{proof}

					We are ready to state an immediate regularity of the constant introduced in \eqref{equation-optimal constant} under  pmGH convergent sequences. 
					\begin{proposition}[Lower semicontinuity of ${\sf C}$]\label{proposition-lower}
						Let 
						$
						(X_i, {\sf d}_i, {\sf m}_i, x_i) \to (X, {\sf d}, {\sf m}, x)
						$
						be a pmGH convergent sequence
						of pointed ${\sf RCD}(K, \infty)$ spaces for some $K \in \mathbb{R}$.
						Then
						$$
						\liminf_{i \to \infty}{\sf C}(X_i, {\sf d}_i, {\sf m}_i, p_i, q_i, t_i) \ge  {\sf C}(X, {\sf d}, {\sf m}, p, q, t)
						$$
						for all $p_i \to p, q_i \to q$ in $[1, \infty]$, and $t_i \to t$ in $(0, \infty)$.
					\end{proposition}
					\begin{proof}
						It is enough to discuss the case when
						\begin{equation}\label{equation-bound}
							\sup_i {\sf C}(X_i, {\sf d}_i, {\sf m}_i, p_i, q_i, t_i) <\infty.
						\end{equation}
						Thanks to  Ambrosio and Honda \cite{AH} (see also  Ambrosio, Bru\`e, and Semola \cite{ABS} and Gigli, Mondino, and Savar\'e \cite{GMS}), for any $f \in L^1\cap L^p \cap L^{\infty}(X, {\sf m})$ with $f \not \equiv 0$, we can find an equibounded sequence of $f_i \in L^1\cap L^{p_i} \cap L^{\infty}(X_i, {\sf m}_i)$ satisfying that  $f_i$ $L^{1}$-strongly converge to $f$ and 
						$$
						\limsup_{i \to \infty}\|f_i\|_{L^{p_i}(X_i, {\sf m}_i)} \le \|f\|_{L^{p}(X, {\sf m})}.
						$$
						On the other hand, (\ref{equation-bound}) implies that ${\sf H}_{t_i}f_i$ is  equi-$L^{q_i}\cap L^{\infty}$ bounded.
						In particular, the lower semicontinuity of $L^{q_i}$-norms with respect to weak convergence yields
						$$
						\liminf_{i \to \infty}\|{\sf H}_{t_i}f_i\|_{L^{q_i}(X_i, {\sf m}_i)} \ge \|{\sf H}_tf\|_{L^{q}(X, {\sf m})}.
						$$
						Therefore, since one has
						$$
						\liminf_{i \to \infty}{\sf C}(X_i, {\sf d}_i, {\sf m}_i, p_i, q_i, t_i) \ge \liminf_{i \to \infty}\frac{\|{\sf H}_{t_i}f_i\|_{L^{q_i}(X_i, {\sf m}_i)}}{\|f_i\|_{L^{p_i}(X_i, {\sf m}_i)}}\ge \frac{\|{\sf H}_tf\|_{L^q(X, {\sf m})}}{\|f\|_{L^p(X, {\sf m})}},
						$$
						taking the supremum with respect to $f$ completes the proof.
					\end{proof}
					
					\begin{remark}\label{remark:ex}
						\rm 
						The above semicontinuity is \textit{sharp} in the following sense. Let $(M^{n-1}, g)$ be a closed Riemannian manifold of dimension $n-1$ with $\mathrm{Ric}\ge n-2$ which is not isometric to $\mathbb{S}^{n-1}(1)=\{x \in \mathbb{R}^n| |x|=1\}$, and consider the $n$-Euclidean cone $X=C(M^{n-1})$. Then
						${\sf AVR}(X)<1$. On the other hand, taking a divergent sequence $x_i \in X$ from the pole along a fixed ray, it is trivial that the sequence $(X, x_i)$ pmGH converges to $\mathbb{R}^n$. Combining this observation with Proposition \ref{proposition-lower}, we have a strict lower semicontinuity of the optimal constants with respect to the pmGH convergence of $(X, x_i)$ to $\mathbb{R}^n$.
						
						%
					\end{remark}

					Based on this remark, let us provide the continuity of ${\sf C}$ under an auxiliary, mild assumption. 
					\begin{proposition}[Continuity to Euclidean cone]\label{proposition-continuity}
						Let 
						$
						(X_i, {\sf d}_i, {\sf m}_i, x_i) \to (X, {\sf d}, {\sf m}, x)
						$
						be a pmGH convergence of pointed ${\sf RCD}(0, N)$ spaces for some $N \in [1, \infty)$. If $X$ is the $N$-Euclidean cone over an ${\sf RCD}(N-2, N-1)$ space and
						$$
						{\sf AVR}(X_i)\to {\sf AVR}(X)>0,
						$$
						then, for all $p_i \to p, q_i \to q$ in $[1, \infty]$, and $t_i \to t$ in $(0, \infty)$,
						$$
						{\sf C}(X_i,  {\sf d}_i, {\sf m}_i, p_i, q_i, t_i) \to {\sf C}(X,  {\sf d}, {\sf m}, p, q, t)
						$$
						whenever $p_i \le q_i$ for any $i$ $($thus $p \le  q)$.
					\end{proposition}
					\begin{proof}
						The lower semicontinuity is justified by Proposition \ref{proposition-lower}, while the upper semicontinuity is guaranteed by \eqref{eq:optimal} and Proposition \ref{proposition-metric cone}.
					\end{proof}
					We are ready to prove the sharpness (\ref{sharp operator norm limit}); to complete this, we consider for some $r>0$ and  $\tau>0$ the rescaled space 
					$
					(X, r{\sf d}, \tau{\sf m}). 
					$
					Then it is easy to see that
					$$
					{\sf AVR}(X, r{\sf d}, \tau{\sf m})=\frac{\tau}{r^N}{\sf AVR}(X, {\sf d}, {\sf m})
					\ \ {\rm	and}\ \  
					{\sf C}(X, r{\sf d}, \tau {\sf m}, p, q, t)=\tau^{\frac{1}{q}-\frac{1}{p}}{\sf C}(X, {\sf d}, {\sf m}, p, q, r^{-2}t).
					$$
					In particular, for every $t,r>0$, we have
					\begin{equation}\label{equation-avr}
						{\sf AVR}\left(X, {r^{-1}}{\sf d}, {r^{-N}}{\sf m}\right)={\sf AVR}(X, {\sf d}, {\sf m})
					\end{equation}
					and
					\begin{equation}\label{equation-rescaling}
						{\sf C}\left( X, {t^{-\frac{1}{2}}}{\sf d}, {t^{-\frac{N}{2}}}{\sf m}, p, q, 1\right)=t^{\frac{N}{2}\left(\frac{1}{p}-\frac{1}{q}\right)}{\sf C}(X, {\sf d}, {\sf m}, p, q, t).
					\end{equation}
					
					\begin{proof}[Proof of \eqref{sharp operator norm limit}]
						For any sequence $t_i \to \infty$, we consider  the rescaled pointed ${\sf RCD}(0,N)$ spaces
						$
						\left(X, {\sf d}_i,  {\sf m}_i, x\right) 
						$
						with ${\sf d}_i:={t_i^{-\frac{1}{2}}}{\sf d}$ and ${\sf m}_i:={t_i^{-\frac{N}{2}}}{\sf m}$. Due to De Philippis and Gigli \cite{DG0}, $
						\left(X, {\sf d}_i,  {\sf m}_i, x\right) 
						$
						pointed measured Gromov-Hausdorff converge, after passing to a subsequence,  to an $N$-Euclidean  cone $
						\left(Y, {\sf d}_0,  {\sf m}_0, x_0\right) 
						$ over an ${\sf RCD}(N-2, N-1)$ space, with ${\sf AVR}(Y, {\sf d}_0,  {\sf m}_0)={\sf AVR}(X,{\sf d},{\sf m})$. Thus, applying Propositions \ref{proposition-metric cone} and \ref{proposition-continuity} together with (\ref{equation-avr}) and (\ref{equation-rescaling}), the proof is complete. 
					\end{proof}

					We conclude this subsection with a monotonicity result and a two-sided bound for  the constant ${\sf C}$; to do this, if $(X, {\sf d}, {\sf m})$ is an ${\sf RCD}(0,N)$ space for some $N \in [1, \infty)$,  we recall 
					the \textit{volume density of  $x \in X$ of order $N$}:
					\begin{equation}
						\nu_{x, N}:=\lim_{r \to 0^+}\frac{{\sf m}(B_r(x))}{\omega_Nr^N}\in (0, \infty].
					\end{equation}
					It is clear that ${\sf AVR}(X) \le \nu_{x, N}$ for every $x\in X$.

					\begin{theorem}[Monotonicity in the case $p=1$ and $q=\infty$]\label{theorem-monoton}
						Let $(X, {\sf d}, {\sf m})$ be an ${\sf RCD}(0,N)$ space for some $N \in [1, \infty)$. Then the following two statements hold$:$ 
						\begin{enumerate}
							\item[(i)] the function
							$	t \mapsto t^{\frac{N}{2}}{\sf C}(X, {\sf d}, {\sf m}, 1, \infty, t)$
							is  non-decreasing on $(0,\infty);$
							\item[(ii)]  for any $t>0$, we have 
							\begin{equation}\label{two-sided}
								\frac{1}{\inf_{x \in X}\nu_{x, N}}\le (4\pi t)^{\frac{N}{2}}{\sf C}(X, {\sf d}, {\sf m}, 1, \infty, t) \le \frac{1}{{\sf AVR}(X)}.
							\end{equation}
						\end{enumerate}
					\end{theorem}
					\begin{proof}
						(i) It follows  from Jiang \cite[ Corollary 6.1]{Jiang} that for all $x, y \in X$, the function 
						$
						t \mapsto	t^{\frac{N}{2}}{\sf h}(x,y,t)
						$
						is  non-decreasing on $(0,\infty)$.
						Thus, for all $f \ge 0$ and $0<s<t$, we have for ${\sf m}$-a.e.\ $x \in X$ that
						\begin{equation}
							t^{\frac{N}{2}}{\sf H}_tf(x)=\int_Xt^{\frac{N}{2}}{\sf h}(x,y,t)f(y){\rm d}{\sf m} \ge \int_Xs^{\frac{N}{2}}{\sf h}(x,y,s)f(y){\rm d}{\sf m}=s^{\frac{N}{2}}{\sf H}_sf(x),
						\end{equation}
						which proves the desired monotonicity.
						
						(ii) The right hand side of \eqref{two-sided} directly follows from \eqref{eq:optimal}. 
						Next, fix $x \in X$ with $\nu_{x, N}<\infty$  and consider the rescaling
						$
						\left( X, {s^{-\frac{1}{2}}}{\sf d}, {s^{-\frac{N}{2}}}{\sf m} \right);
						$
						after passing to a subsequence, take a pmGH limit of the above as $s \to 0^+$. Then the limit space $(Y, {\sf d}_Y, {\sf m}_Y)$, called a tangent cone at $x$, is the $N$-Euclidean cone over an ${\sf RCD}(N-2, N-1)$ space with ${\sf AVR}(Y)=\nu_{x, N}$. Now, the monotonicity from (i), relation \eqref{equation-rescaling}, Proposition \ref{proposition-lower}, and Corollary \ref{cor:explicit} allow  us to conclude for every $t>0$ that
						\begin{eqnarray*}
							(4\pi t)^{\frac{N}{2}}{\sf C}(X, {\sf d}, {\sf m}, 1, \infty, t)  &\ge& \lim_{s \to 0^+}(4\pi s)^{\frac{N}{2}}{\sf C}(X, {\sf d}, {\sf m}, 1, \infty, s)\\&=&(4\pi )^{\frac{N}{2}}\lim_{s \to 0^+}{\sf C}(X, s^{-\frac{1}{2}}{\sf d}, s^{-\frac{N}{2}}{\sf m}, 1, \infty, s) \\&\ge& (4\pi )^{\frac{N}{2}}{\sf C}(Y, {\sf d}_Y, {\sf m}_Y, 1, \infty, 1)=\frac{1}{{\sf AVR}(Y)}\\&=&\frac{1}{\nu_{x, N}}.
						\end{eqnarray*}
						Since $x$ is arbitrary, the left hand side of \eqref{two-sided} also follows. 
					\end{proof}

					\begin{corollary}
						Let $(X, {\sf d}, {\sf m})$ be a non-collapsed ${\sf RCD}(0, N)$ space with no singular points. Then
						\begin{equation}\label{eq:two-sided}
							\frac{1}{(4\pi t)^{\frac{N}{2}}}\le \|{\sf H}_t\|_{1, \infty} \le \frac{{\sf AVR}(X)^{-1}}{(4\pi t)^{\frac{N}{2}}}, \quad \forall t>0.
						\end{equation}
					\end{corollary}

					\subsection{Sharp large time behavior of the heat kernel} 
					
					Given a metric space $(X,{\sf d})$, we say that the path $t\mapsto (x(t),y(t),t)\in X\times X\times (0,\infty)$, $t>0$, verifies the \textit{Li-growth} if, for some fixed $\tilde x\in X$, 
					\begin{equation}\label{Li-growth}
						{\sf d}^2(\tilde x,x(t))+{\sf d}^2(\tilde x,y(t)) =o(t)\ \ {\rm as}\ \ t\to \infty.
					\end{equation}
					
					Although the following result is known --   initially   proved by 
					Li \cite[Theorem 1]{Li-Annals} on Riemannian manifolds with non-negative Ricci curvature, and then extended to ${\sf RCD}(0,N)$  spaces by Jiang \cite[Theorem 1.4]{Jiang} and  Jiang, Li, and Zhang \cite{JLZ}, -- we provide here a simple proof by using a blow-down rescaling argument. 
					
					\begin{theorem}[Sharp large time behavior of the heat kernel]\label{theorem-Jiang-Li} Given $N>1$,
						let $(X,{\sf d},{\sf m})$ be an   ${\sf RCD}(0,N)$ metric measure space with ${\sf AVR}(X) >0.$ Then, for every $z\in X$ and every path $t\mapsto (x(t),y(t),t)\in X\times X\times (0,\infty)$, $t>0$, verifying  the Li-growth, one has that 
						\begin{equation}\label{limit-large-1}
							\lim_{t\to \infty} {\sf m}(B_{\sqrt{t}}(z)){\sf h}(x(t),y(t),t)=\frac{\omega_N}{(4\pi)^\frac{N}{2}}.
						\end{equation}
					\end{theorem}
					
					\begin{proof}
							Let $z\in X$ be fixed, and let $t\mapsto (x(t),y(t),t)\in X\times X\times (0,\infty)$, $t>0$, be any path verifying  the Li-growth \eqref{Li-growth} for some $\tilde x\in X.$
							Let $t_i \to \infty$ be any sequence, and consider  the rescaled pointed ${\sf RCD}(0,N)$ spaces
							$
							\left(X, {\sf d}_i,  {\sf m}_i, \tilde x\right) 
							$
							with $${\sf d}_i:={t_i^{-\frac{1}{2}}}{\sf d}\ \ {\rm and}\ \ {\sf m}_i:={\sf m}(B_{\sqrt{t_i}}(z))^{-1}{\sf m}.$$  As before, the pointed spaces $
							\left(X, {\sf d}_i,  {\sf m}_i, \tilde x\right) 
							$
							pmGH converge  to an $N$-Euclidean  cone $
							\left(Y, {\sf d}_0,  {\sf m}_0, x_0\right) 
							$ over an ${\sf RCD}(N-2, N-1)$ space, where  $x_0\in Y$ is a tip of the cone and ${\sf AVR}(Y, {\sf d}_0,  {\sf m}_0)=\omega_N^{-1}$; here, we used the fact that ${\sf AVR}(X)>0$. We recall that on the $N$-Euclidean  cone $
							\left(Y, {\sf d}_0,  {\sf m}_0, x_0\right) 
							$ the heat kernel is given by \eqref{heat-kernel-on-cone}, i.e., 
							$$	{\sf h}_Y(x_0, x, s)=\frac{\omega_N}{(4\pi s)^{N/2}}e^{-\frac{{\sf d}_0^2(x_0,x)}{4s}},\ \ x\in Y.$$
							If ${\sf h}_i$ stands for the heat kernel on $
							\left(X, {\sf d}_i,  {\sf m}_i\right), 
							$ then the above rescaling gives 
							\begin{equation}\label{heat-scale}
								{\sf m}(B_{\sqrt{t_i}}(z)){\sf h}(x(t_i),y(t_i),t_i)={\sf h}_i(x(t_i),y(t_i),1).
							\end{equation}
							Note that ${\sf d}_i(x(t_i),y(t_i))\to 0$ as $i\to \infty;$ indeed, by the Li-growth assumption one has that
							\begin{equation}\label{Li-asymptotic}
								{\sf d}_i^2(x(t_i),y(t_i))=\frac{1}{{t_i}}{\sf d}^2(x(t_i),y(t_i))\leq {2} \frac{{{\sf d}^2(\tilde x,x(t_i))+{\sf d}^2(\tilde x,y(t_i))}}{{t_i}}=o(1)\ \ {\rm as}\ i\to \infty.
							\end{equation}
							Note that the latter fact means that both $x(t_i)$ and $y(t_i)$   pmGH converge to the  tip $x_0$ of the cone $Y$. 
							Therefore, by relation \eqref{heat-scale} and the pointwise convergence of heat kernels, see Ambrosio, Honda, and Tewodrose \cite[Theorem 3.3]{AHT}, it follows that 
							\begin{equation}\label{diagonal-Li}
								\lim_{i\to \infty} {\sf m}(B_{\sqrt{t_i}}(z)){\sf h}(x(t_i),y(t_i),t_i)=\lim_{i\to \infty}{\sf h}_i(x(t_i),y(t_i),1) =  {\sf h}_Y(x_0,x_0,1)=\frac{\omega_N}{(4\pi )^{N/2}},
							\end{equation}
							which concludes the proof.  \end{proof}
						
						
						\begin{remark}\rm 
							(i) By \eqref{limit-large-1}, for any fixed points $x,y,z\in X$, one has 
							$$	\lim_{t\to \infty} {\sf m}(B_{\sqrt{t}}(z)){\sf h}(x,y,t)=\frac{\omega_N}{(4\pi)^\frac{N}{2}}.$$
							Indeed, we may choose in \eqref{Li-growth} the constant paths $x(t)=x$, $y(t)=y$ for every $t>0$.    
							
							(ii) As we have noticed, relation \eqref{Li-growth} played a crucial role in 
							the arguments of	Li \cite[Theorem 1]{Li-Annals} and Jiang \cite[Theorem 1.4]{Jiang}, where the authors used the Harnack  inequality to prove the large time behavior of the heat kernel \eqref{limit-large-1}. It is interesting to point out the importance of the same growth assumption \eqref{Li-growth} in our approach, reflected in proving \eqref{Li-asymptotic} and using the blow-down rescaling argument together with the pmGH convergence, which is a proof that is genuinely independent from Li \cite{Li-Annals} and Jiang \cite{Jiang}. It is clear that \eqref{Li-growth} cannot be improved, as the classical Euclidean setting already provides  counterexamples.    
						\end{remark}
						
						A direct consequence of Theorem \ref{theorem-Jiang-Li}, together with Honda \cite{H24}, reads as follows (compare also with (\ref{eq:two-sided})): 
						
						\begin{corollary}
							Under the same assumptions as in Theorem \ref{theorem-Jiang-Li}, one has that 
							$$	\frac{\min\{\nu_{x, N}, \nu_{y, N}\}^{-1}}{(4\pi)^{\frac{N}{2}}} e^{-\frac{{\sf d}^2(x,y)}{4t}}\le  t^\frac{N}{2}{\sf h}(x,y,t)\leq  \frac{{\sf AVR}(X)^{-1}}{(4\pi)^\frac{N}{2}},\ \ \forall x,y\in X,\ \forall t>0.$$
							Moreover, both estimates are sharp, in the sense that 
							$$
							\lim_{t\to \infty}	t^\frac{N}{2}{\sf h}(x,y,t)= \frac{{\sf AVR}(X)^{-1}}{(4\pi)^\frac{N}{2}},  \quad \forall x,y \in X
							$$
							and
							$$
							\lim_{t \to 0^+}t^\frac{N}{2}{\sf h}(x,x,t)=\frac{\nu_{x, N}^{-1}}{(4\pi)^{\frac{N}{2}}}, \quad \forall x \in X.
							$$
						\end{corollary}
						%
						

						%
						%
						%
						
						We conclude this subsection with certain characterizations concerning fine properties of the heat kernel on ${\sf RCD}(0,N)$ spaces. 
						
						\begin{corollary}
							Let $(X, {\sf d}, {\sf m})$ be an ${\sf RCD}(0,N)$ space for some $N \in [1, \infty)$. Then we have the following statements. 
							\begin{enumerate}
								\item[(i)] For fixed $x \in X$, the following two conditions are equivalent$:$
								\begin{enumerate}
									\item the function $t \mapsto t^{\frac{N}{2}}{\sf h}(x,x,t)$ is constant$;$
									\item $(X, {\sf d}, {\sf m})$ is the $N$-Euclidean cone over an ${\sf RCD}(N-2, N-1)$ space with the tip $x$.
								\end{enumerate}
								\item[(ii)] The following two conditions are equivalent$:$
								\begin{enumerate}
									\item for any $x \in X$, the function $t \mapsto t^{\frac{N}{2}}{\sf h}(x,x,t)$ is constant;
									\item $N$ is an integer and $(X, {\sf d}, {\sf m})$ is the $N$-dimensional Euclidean space, up to  the multiplication of ${\sf m}$ by a positive constant.
								\end{enumerate}
							\end{enumerate}
						\end{corollary}
						\begin{proof}
							(i) Assume (a). Then we know that $\nu_{x, N}={\sf AVR}(X)$; indeed, in a similar manner as in the proof of Theorem \ref{theorem-Jiang-Li} -- by using the short time behavior of the heat kernel with a blow-down rescaling as in the proof of Theorem \ref{theorem-monoton}/(ii) -- the claimed relation follows. 
							Now, from De Philippis and Gigli \cite{DG0}, we have (b). The converse is trivial, see \eqref{heat-kernel-on-cone}.  
							
							(ii) Assume (a). Then, due to (i), for any $x \in X$, $X$ is an $N$-Euclidean cone with the tip $x$. In particular, for all $x, y \in X$ with $x \neq y$, we can find a unique line passing through $x$ and $y$. Applying the splitting theorem iteratively, we conclude that $X$ is an Euclidean space. Recalling that $\nu_{x,N}={\sf AVR}(X)$ for any $x \in X$, the dimension is necessarily $N$, which ends the proof of (b). The converse is again trivial. 
						\end{proof}
						
						\begin{remark}
							\rm 
							The power $\frac{N}{2}$ of $t$ appearing in our results  cannot be replaced by any other exponent; indeed,  there are  ${\sf RCD}(0, N)$ spaces with a \textit{wild} behavior of their heat kernels. For example, there exists a pointed ${\sf RCD}(0,N)$ space $(X, {\sf d}, {\sf m}, x)$ for some $N>2$, called \textit{a Grushin half plane}, such that the function $t \mapsto t^{\frac{\alpha}{2}}{\sf h}(x,x,t)$ is constant for some $0<\alpha<N$, but  $(X, {\sf d}, {\sf m}, x)$ is {\it not} isometric to any Euclidean cone, see Dai, Honda, Pan, and Wei \cite[Lemma 3.16]{DHPW}.
						\end{remark}

						\section{Equality case: proof of Theorem \ref{theorem-equality-RCD}}\label{section-4}
						
						\subsection{(i) $\implies$ (ii) of Theorem \ref{theorem-equality-RCD} \& the uniqueness of the extremizer}\label{subsection-rigidity} 
						Assume that  equality holds in \eqref{Heat-bound-main} for some  $1<  p< q< \infty$,  $t_0>0$, and a non-negative non-zero $f\in L^1\cap L^\infty(X,{\sf m})$. We are going to trace back the estimates in the proof of  \eqref{Heat-bound-main}, identifying the equalities. 
						
						Keeping the notation as before, there exists a unique $\lambda_0>0$ such that  $t(\lambda_0)=t_0$, and according to 
						\eqref{lambdatlambda}, it follows that 
						\begin{equation}\label{lambdatlambda-0}
							\lambda_0=\frac{N}{8t_0}\left(\frac{1}{p}-\frac{1}{q}\right),
						\end{equation}
						Moreover, $q=p(t(\lambda_0))=p(t_0)$ and the function $v$ in \eqref{lambdav} is with the parameter $\lambda_0,$ i.e., 
						$v(s)=\lambda_0 \frac{s^2}{s-1},$ $s>1.$ The crucial point is that the equality should hold in \eqref{log-Sobolev-third} for the function 
						\begin{equation}\label{u-egyenloseg}
							u=f_{t_0}^{p(t_0)/2}=f_{t_0}^{q/2}.
						\end{equation}
						In fact, this requires the following simultaneous \textit{equalities}: 
						\begin{itemize}
							\item in the concavity inequality \eqref{concave} for the corresponding choices. Since $\Phi$ is strictly concave, we necessarily have 
							\begin{equation}\label{eqgy-u-es-v}
								\frac{{\ds\int_X \abs{\nabla u}^2 \, {\rm d}{\sf m}}}{{\|u\|^2_{L^2(X,{\sf m})}}}= v(p(t_0));
							\end{equation}
							\item in the logarithmic Sobolev inequality \eqref{log-Sobolev-second} for the function $u$ from \eqref{u-egyenloseg}. 
						\end{itemize}
						According to the equality case in the logarithmic Sobolev inequality \eqref{log-Sobolev-second} (see \S \ref{section-log-sob}), the latter statement implies that $(X,{\sf d},{\sf m})$ is an $N$-Euclidean cone with a tip $x_0\in X$ and for some $c_1\in \mathbb R$ and $c_0>0$ one has
						\begin{equation}\label{f-es-u-kombo}
							f_{t_0}^{q/2}(x)=u(x)=c_1 e^{-c_0{\sf d}^2(x_0,x)}\ \ {\rm for}\  {\sf m}-{\rm a.e}\ x\in X.
						\end{equation}

						On the other hand, by using  the fact that on the $N$-Euclidean cone $(X,{\sf d},{\sf m})$ with the tip $x_0\in X$ one has that 
						\begin{equation}\label{volume-cone-ball}
							\frac{ {\sf m}(B(x_0,r))}{\omega_Nr^N}={\sf AVR}(X), \quad \forall r>0,
						\end{equation}
						the layer cake representation implies  
						\begin{align*}
							\|u\|^2_{L^2(X,{\sf m})}&=\int_X u^2 {\rm d}{\sf m}=c_1^2\int_X u^2 {\rm d}{\sf m}=c_1^2\int_X e^{-2c_0{\sf d}^2(x_0,x)} {\rm d}{\sf m}
							\\&=c_1^2\left(\frac{\pi}{2c_0}\right)^{N/2}{\sf AVR}(X).
						\end{align*}
						Since $$\nabla u=-2c_0c_1  e^{-c_0{\sf d}^2(x_0,x)}{\sf d}(x_0,x)\nabla {\sf d}(x_0,x),$$ the eikonal equation $|\nabla {\sf d}(x_0,x)|=1$ for ${\sf m}$-a.e.\ $x\in X$ together with an argument similar to the one above shows that 
						$$\int_X \abs{\nabla u}^2 \, {\rm d}{\sf m}=4c_0^2c_1^2\int_X e^{-2c_0{\sf d}^2(x_0,x)}{\sf d}^2(x_0,x) {\rm d}{\sf m}=Nc_0c_1^2\left(\frac{\pi}{2c_0}\right)^{N/2}{\sf AVR}(X).$$
						
						Therefore, combining these facts with  \eqref{eqgy-u-es-v} and \eqref{lambdav}, it follows that $$Nc_0=v(p(t_0))=v(q)=\lambda_0 \frac{q^2}{q-1}.$$
						Accordingly, by \eqref{f-es-u-kombo}, \eqref{lambdatlambda-0}, and the latter relation, it follows that, up to a multiplicative constant,
						\begin{equation}\label{f-t-0}
							{\sf H}_{t_0}f(x)=f_{t_0}(x)= e^{-\beta_0{\sf d}^2(x_0,x)}\ \ {\rm for}\  {\sf m}-{\rm a.e}\ x\in X,
						\end{equation} 
						where 
						\begin{equation}\label{alpha-def}
							\beta_0:=\frac{2c_0}{q}=\frac{1}{4t_0}\frac{q}{q-1}\left(\frac{1}{p}-\frac{1}{q}\right).
						\end{equation} 
						
						Let us find  $C>0$ and $\tilde t>0$ satisfying
						$$
						C {\sf H}_{t_0}{\sf h}\left(x_0, x, \tilde t\right)=C  {\sf h}\left(x_0, x, t_0+\tilde t\right)=e^{-\beta_0{\sf d}^2(x_0,x)}, \ \forall x\in X.
						$$
						By \eqref{heat-kernel-on-cone}, it turns out that 
						$\tilde t=\frac{1}{4\beta_0}-t_0=t_0\frac{q(p-1)}{q-p}$ and $C=\left(\frac{\pi}{\alpha_0}\right)^\frac{N}{2}{\sf AVR}(X).$
						Due to \eqref{f-t-0} and the latter relation, it follows that ${\sf H}_{t_0}f={\sf H}_{t_0}(C{\sf h}(x_0, \cdot, \tilde t))$. By Proposition \ref{proposition-unique},  up to a multiplicative constant, one has that  $f=C{\sf h}(x_0, \cdot, \tilde t)$, i.e.,
						$f=e^{-\alpha_0 {\sf d}^2(x_0,\cdot)}$  for
						${\sf m}$-a.e.\ on $X,$ where 
							\begin{equation}\label{alpha-good}
								\alpha_0:=\frac{1}{4\tilde t}=\frac{1}{4t_0}\frac{p}{p-1}\left(\frac{1}{p}-\frac{1}{q}\right)>0,
							\end{equation}
							which concludes the proof. 
							
							The implication (ii) $\implies$ (i) is trivial, based again on  \eqref{heat-kernel-on-cone} and elementary computations.

							%
							%

							\subsection{Alternative direct proof of Corollary \ref{corollary-equality-Riemannian-1}}\label{section-altenative}
							Assume that we have equality in \eqref{Heat-bound-main} for some  $1<  p< q< \infty$,  $t_0>0$, and a non-negative  non-zero $f\in L^1(M,v_g)\cap L^\infty(M,v_g)$.
							According to Theorem \ref{theorem-equality-RCD}, $(M,{\sf d}_g,v_g)$ is an $n$-Euclidean cone (with a tip $x_0\in M)$ and ${\sf H}_{t_0}f=f_{t_0}$ is given by \eqref{f-t-0}. Due to  \eqref{volume-cone-ball}  and Gallot,  Hulin, and Lafontaine \cite[Theorem 3.98]{GHL}, it follows that  $${\sf AVR}(M)=\lim_{r\to \infty}\frac{v_g(B_r(x_0))}{\omega_nr^n}=\lim_{r\to 0}\frac{v_g(B_r(x_0))}{\omega_nr^n}=1,$$ i.e., ${\sf AVR}(M)=1$, which implies that $(M,g)$ is isometric to $\mathbb R^n$. In particular, the heat kernel on $(M,g)$ has -- up to isometry -- the form 
							\begin{equation}\label{heat-kernel-R^n}
								{\sf h}(x,y, t)= \frac{1}{(4\pi t)^{n/2}}e^{-\frac{|x-y|^2}{4t}}, \ \ \ t>0,\ x,y\in M\simeq \mathbb R^n.
							\end{equation}
							Relation \eqref{heat-kernel-R^n} and  \eqref{heat-integral-representation} imply that
							\begin{equation}\label{convolution}
								{\sf H}_{t_0} f(x)=\int_{\mathbb R^n} {\sf h}(x,y, t_0)f(y){\rm d}y=(P_{t_0}*f)(x), \ \  x\in \mathbb R^n,
							\end{equation}
							where $P_{t_0}(z)={\sf h}(z,0, t_0)$, $z\in \mathbb R^n$, and $'*'$ stands for the convolution. By using the Fourier transform 
							$$\mathcal F h(\xi)=\frac{1}{(2\pi)^{n/2}}\int_{\mathbb R^n}e^{-i \xi\cdot x}h(x){\rm d}x$$
							for a Schwartz-type function $h:\mathbb R^n\to \mathbb R$ (where $\cdot$ stands for the usual inner product in $\mathbb R^n$), the convolution  \eqref{convolution} is transformed into the multiplicative relation
							$$\mathcal F	({\sf H}_{t_0} f)(\xi)=(2\pi)^{n/2}\mathcal F P_{t_0}(\xi) \mathcal Ff(\xi),\ \ \xi\in \mathbb R^n,$$
							see Reed and Simon \cite[Theorem IX.3]{Reed-Simon}. 
							If $h(x)=e^{-c|x-x_0|^2}$ for some $c>0$ and $x_0\in \mathbb R^n$, it is standard to show that 
							$$\mathcal F h(\xi)=\frac{1}{(2c)^{n/2}}e^{-\frac{|\xi|^2}{4c}-i \xi\cdot x_0},\ \ \xi\in \mathbb R^n.$$
							This relation and \eqref{f-t-0}  show  that 
							$$\mathcal F	({\sf H}_{t_0} f)(\xi)=\frac{1}{(2\beta_0)^{n/2}}e^{-\frac{|\xi|^2}{4\beta_0}-i \xi\cdot x_0},\ \ \xi\in \mathbb R^n,$$
							where $\beta_0>0$ is given by \eqref{alpha-def}. In a similar manner, one has that 
							$$\mathcal F P_{t_0}(\xi)=\frac{1}{(2\pi)^{n/2}}e^{-t_0|\xi|^2},\ \ \xi\in \mathbb R^n.$$
							Combining these expressions, it turns out that
							$$\mathcal Ff(\xi)=\frac{1}{(2\beta_0)^{n/2}}e^{-\left(\frac{1}{4\beta_0}-t_0\right)|\xi|^2-i \xi\cdot x_0},\ \ \xi\in \mathbb R^n.$$
							Note that $\frac{1}{4\beta_0}-t_0>0$ as $p>1$, see \eqref{alpha-def}. 
							Now, using the inverse Fourier transform, we obtain
							\begin{equation}\label{f-gauss}
								f(x)=\frac{1}{(1-4\beta_0 t_0)^{n/2}}e^{-\alpha_0|x-x_0|^2},\ \ x\in \mathbb R^n,
							\end{equation}
							where 
							$\alpha_0=\frac{\beta_0}{1-4\beta_0 t_0}
							$
							is precisely the value from \eqref{alpha-good}. The converse is trivial. 
							

							\section{Rigidity and almost rigidity}\label{section-rigidity}
							The main purpose of this section is to provide some (almost) rigidity results. 
							We start by proving Li's rigidity on non-collapsed spaces: \\
							
							{\it Proof of Theorem \ref{Li-rigidity}.}
							It is enough to prove the implication from (i) to (iii). Since 
							$$
							\limsup_{t \to \infty} \left((4\pi t)^{\frac{N}{2}(\frac{1}{p}-\frac{1}{q})}\|{\sf H}_t\|_{p, q} \right)\le {\sf M}(p,q)^\frac{N}{2},
							$$
							the first part of Theorem \ref{main-theorem-RCD} implies that ${\sf AVR}(X)>0$. Now, applying  \eqref{sharp operator norm limit}, it turns out that  ${\sf AVR}(X)\geq 1$. On the other hand, since $(X, {\sf d}, \mathcal{H}^N)$ is a non-collapsed  ${\sf RCD}(0,N)$ space, we have that ${\sf AVR}(X)\leq 1$. Therefore, ${\sf AVR}(X)= 1$, which is equivalent to the fact that $X$ is isomeric to $\mathbb{R}^N$. \hfill $\square$\\
							
							In the sequel we provide a quantitative version of Proposition \ref{proposition-continuity}.

							\begin{proposition}[Almost optimality from almost Euclidean cone]\label{proposition-almost-optimal}
								For all $\epsilon>0$, $0<\tau<1$, and $N \in [1, \infty)$, there exists $\delta=\delta(\epsilon,  \tau, N)>0$ such that if 
								an ${\sf RCD}(0, N)$ space $(X, {\sf d}, {\sf m})$ with ${\sf AVR}(X)>0$ satisfies 
								$$
								\frac{\nu_{x, N}}{{\sf AVR}(X)} \le 1+\delta, \quad \tau \le {\sf m}(B_1(x)) \le \frac{1}{\tau}
								$$
								for some $x \in X$, then 
								$$
								\left| t^{\frac{N}{2}\left(\frac{1}{p}-\frac{1}{q}\right)}{\sf C}(X,  {\sf d}, {\sf m}, p, q, t)-\frac{{\sf M}(p,q)^\frac{N}{2}{\sf AVR}(X)^{\frac{1}{q}-\frac{1}{p}}}{(4\pi )^{\frac{N}{2}(\frac{1}{p}-\frac{1}{q})}}\right| \le \epsilon
								$$
								for all $t \in (0, \infty)$ and $p, q \in [1, \infty]$ with $p\le q$. 
							\end{proposition}
							\begin{proof}
								Assume by contradiction that there exist: 
								\begin{itemize}
									\item a sequence of pointed ${\sf RCD}(0, N)$ spaces $(X_i, {\sf d}_i, {\sf m}_i, x_i)$ with
									\begin{equation}\label{limit-avr}
										\frac{\nu_{x_i, N}}{{\sf AVR}(X_i)} \to 1, \quad \tau \le {\sf m}_i(B_1(x_i)) \le \frac{1}{\tau};
									\end{equation}
									\item sequences of $t_i \in (0, \infty)$ and $p_i, q_i \in [1, \infty]$ with $p_i\le q_i$ and
									\begin{equation}\label{equation-positive}
										\liminf_{i \to \infty}\left| t_i^{\frac{N}{2}\left(\frac{1}{p_i}-\frac{1}{q_i}\right)}{\sf C}(X_i,  {\sf d}_i, {\sf m}_i, p_i, q_i, t_i)-\frac{{\sf M}(p_i, q_i)^\frac{N}{2}{\sf AVR}(X_i)^{\frac{1}{q_i}-\frac{1}{p_i}}}{(4\pi )^{\frac{N}{2}(\frac{1}{p_i}-\frac{1}{q_i})}}\right|>0.
									\end{equation}
								\end{itemize}
								After passing to a subsequence, with no loss of generality, we can assume that
								\begin{itemize}
									\item $(X_i, {\sf d}_i, {\sf m}_i, x_i)$ pmGH converge to an ${\sf RCD}(0, N)$ space $(X, {\sf d}, {\sf m}, x)$; 
									\item $t_i \to t$ in $[0, \infty]$, $p_i \to p$, $q_i \to q$ in $[1, \infty]$. 
								\end{itemize} 
								By definition, one has that $\nu_{x, N}\geq {\sf AVR}(X)$. On the other hand, by \eqref{limit-avr}, the lower  semicontinuity of the density $\nu_{\cdot, N}$, and the upper semicontinuity of ${\sf AVR}(\cdot)$ under the pmGH convergence, see Honda and Peng \cite[Lemma 3.10]{Honda-Peng} and Krist\'aly and Mondino \cite[Lemma 6.1]{KriM}, it follows that $\nu_{x, N}\leq {\sf AVR}(X)$. Therefore,  $\nu_{x, N}={\sf AVR}(X)$, which implies that  $(X, {\sf d}, {\sf m}, x)$ is isometric to the $N$-Euclidean cone over an ${\sf RCD}(N-2, N-1)$ space. Now, we distinguish two cases.
								
								\textit{Case 1:} $0<t<\infty$. In this case, due to Propositions \ref{proposition-continuity} and \ref{proposition-metric cone}, the left hand side of \eqref{equation-positive}  converges to zero, a contradiction.
								
								\textit{Case 2:} $t=0$ or $t=\infty$.
								We consider the rescaled spaces
								$
								\left(X_i, {t_i^{-\frac{1}{2}}}{\sf d}_i,  {t_i^{-\frac{N}{2}}}{\sf m}_i, x_i\right)
								$ with $t_i \to 0^+$ or $t_i\to \infty$ as $i\to \infty$, respectively.
								After passing to a subsequence, the above rescaled ${\sf RCD}(0, N)$ spaces pmGH converge to an $N$-Euclidean cone over an ${\sf RCD}(N-2, N-1)$ space. By using the scaling properties (\ref{equation-avr}) and (\ref{equation-rescaling}) together with Propositions \ref{proposition-continuity} and \ref{proposition-metric cone}, we obtain  again a contradiction through \eqref{equation-positive}.  
							\end{proof}

							In the sequel, we focus on almost rigidity results. It should be mentioned that a positive lower bound of $q-p$ as required in the following cannot be removed because the condition (\ref{c-almost-1}) is trivially satisfied without any such assumption in the case when $p=q$.
							\begin{proposition}[Almost rigidity I]\label{proposition-optimal to cone}
								For all $\epsilon>0$, $N \in [1, \infty)$, and $0<\tau<1$, there exists $\delta=\delta(\epsilon, N, \tau)>0$ such that if a pointed ${\sf RCD}(0, N)$ space $(X, {\sf d}, {\sf m}, x)$ satisfies $\tau \le {\sf m}(B_1(x)) \le \tau^{-1}$ and 
								\begin{equation}\label{c-almost-1}
									{\sf C}(X,  {\sf d}, {\sf m}, p, q, t) \le  (1+\delta) \frac{{\sf M}(p,q)^\frac{N}{2}\nu_{x, N}^{\frac{1}{q}-\frac{1}{p}}}{(4\pi t)^{\frac{N}{2}(\frac{1}{p}-\frac{1}{q})}}<\infty, \quad \forall t>1
								\end{equation}
								for some $p \in [1, \frac{1}{\tau}]$ and some $q \in [1, \infty]$ with $q-p \ge \tau$,  
								then 
								$X$ is $\epsilon$-pointed measured Gromov-Hausdorff close to an $N$-Euclidean cone over an ${\sf RCD}(N-2, N-1)$ space with 
								$$
								\left| s^{\frac{N}{2}\left(\frac{1}{\bar p}-\frac{1}{\bar q}\right)}{\sf C}(X,  {\sf d}, {\sf m}, \bar p, \bar q, s)-\frac{{\sf M}(\bar p,\bar q)^\frac{N}{2}{\sf AVR}(X)^{\frac{1}{\bar q}-\frac{1}{\bar p}}}{(4\pi )^{\frac{N}{2}(\frac{1}{\bar p}-\frac{1}{\bar q})}}\right| \le \epsilon
								$$
								for all $s \in (0, \infty)$ and $\bar p, \bar q \in [1, \infty]$ with $\bar p\le \bar q$.
							\end{proposition}
							\begin{proof}
								According to the first part of Theorem \ref{main-theorem-RCD} and \eqref{c-almost-1}, it follows that ${\sf AVR}(X)>0.$ Thus, by  (\ref{sharp operator norm limit}) and \eqref{c-almost-1}, one has that 
								${\sf AVR}(X)$ is close to $\nu_{x, N}$, namely 
								$$(1+\delta)^\frac{pq}{p-q}\nu_{x, N}\leq {\sf AVR}(X)\leq   \nu_{x, N} .$$
								Now, the assertion follows by  De Philippis and Gigli \cite{DG0} and Proposition \ref{proposition-almost-optimal}. 
							\end{proof}
							
							\begin{remark}\rm 
								Note that in Proposition \ref{proposition-optimal to cone} the rigid case, i.e., $\delta=0$,  provides a new characterization of the fact that $(X, {\sf d}, {\sf m}, x)$ is an $N$-Euclidean cone.
							\end{remark}

							\begin{proposition}[Almost rigidity II]\label{prop-almost-rig}
								For all $\epsilon>0$, $0<\tau<1$, and $N \in [1, \infty)$, there exists $\delta=\delta(\epsilon,  \tau, N)>0$ such that the following hold$:$ let $(X, {\sf d}, {\sf m}, x_0)$ be  
								a pointed ${\sf RCD}(0, N)$ space with $\tau \le {\sf m}(B_1(x_0)) \le \tau^{-1}$ and ${\sf AVR}(X)>0$, let $t \in [\tau, \tau^{-1}]$, and let $p, q \in [1, \infty)$ with $1+\tau \le p \le p+\tau \le q \le \tau^{-1}$ and
								\begin{equation}
									\frac{\|{\sf H}_tf\|_{L^q}}{\|f\|_{L^p}}\ge (1-\delta) \frac{{\sf M}(p,q)^\frac{N}{2}{\sf AVR}(X)^{\frac{1}{q}-\frac{1}{p}}}{(4\pi )^{\frac{N}{2}(\frac{1}{p}-\frac{1}{q})}},
								\end{equation}
								where $f=e^{-\alpha {\sf d}^2(\cdot, x_0)}$ for some $\alpha \in [\tau, \tau^{-1}]$.
								Then $(X, {\sf d}, {\sf m}, x_0)$ is $\epsilon$-pmGH close to the $N$-Euclidean cone $(Y, {\sf d}_0, {\sf m}_0, y_0)$ over an ${\sf RCD}(N-2, N-1)$ space. Moreover, $|{\sf AVR}(X)-{\sf AVR}(Y)|\le \epsilon$.
							\end{proposition}
							\begin{proof}
								The proof of the first statement is  similar to that of Proposition \ref{proposition-almost-optimal}. Namely, assume by contradiction that there exist:
								\begin{itemize}
									\item a pmGH convergent sequence of pointed ${\sf RCD}(0, N)$ spaces $(X_i, {\sf d}_i, {\sf m}_i, x_i) \to (X, {\sf d}, {\sf m}, x)$ so that $(X, {\sf d}, {\sf m}, x)$ is not isometric to the $N$-Euclidean cone over any ${\sf RCD}(N-2, N-1)$ space;
									\item convergent sequences $p_i \to p, q_i \to q$ in $[1, \infty)$ with $1<p<q$;
									\item convergent sequences $\alpha_i \to \alpha, c_i \to c$, and $t_i \to t$ in $(0, \infty)$ with $f_i=c_ie^{-\alpha_i{\sf d}_i^2(x_i, \cdot)}$ and $\|f_i\|_{L^{p_i}}=1$ such that
									\begin{equation}\label{eq:avr conv}
										\lim_{i \to \infty}\left| \|{\sf H}_{t_i}f_i\|_{L^{q_i}}-\frac{{\sf M}(p_i,q_i)^\frac{N}{2}{\sf AVR}(X_i)^{\frac{1}{q_i}-\frac{1}{p_i}}}{(4\pi )^{\frac{N}{2}(\frac{1}{p_i}-\frac{1}{q_i})}}\right|= 0.
									\end{equation}
								\end{itemize}
								Notice, by Ambrosio and Honda \cite{AH}, that since $f_i$ $L^1$-strongly converge to $f=ce^{-\alpha {\sf d}^2(x, \cdot)}$ with $1=\lim_{i \to \infty}\|f_i\|_{L^{p_i}}=\|f\|_{L^p}$, we know $\lim_{i \to \infty}\|{\sf H}_{t_i}f_i\|_{L^{q_i}}=\|{\sf H}_tf\|_{L^q}$. On the other hand, recalling the upper semicontinuity of ${\sf AVR}$, we have
								\begin{align*}
									\frac{{\sf M}(p,q)^\frac{N}{2}{\sf AVR}(X)^{\frac{1}{q}-\frac{1}{p}}}{(4\pi )^{\frac{N}{2}(\frac{1}{p}-\frac{1}{q})}}&\le \liminf_{i\to \infty}\frac{{\sf M}(p_i,q_i)^\frac{N}{2}{\sf AVR}(X_i)^{\frac{1}{q_i}-\frac{1}{p_i}}}{(4\pi )^{\frac{N}{2}(\frac{1}{p_i}-\frac{1}{q_i})}} =\lim_{i\to \infty}\|{\sf H}_{t_i}f_i\|_{L^{q_i}}=\|{\sf H}_tf\|_{L^q}.
								\end{align*}
								In particular, $\lim_{i \to \infty}{\sf AVR}(X_i)={\sf AVR}(X)$ because of (\ref{Heat-bound-main}) and (\ref{eq:avr conv}).
								Then Theorem \ref{theorem-equality-RCD} allows us to conclude that $(X, {\sf d}, {\sf m}, x)$ is isometric to the $N$-Euclidean cone over an ${\sf RCD}(N-2, N-1)$ space, which is a contradiction. 
								The  closeness of ${\sf AVR}$ is also a direct consequence of the above arguments.
							\end{proof}

							\begin{proposition}[Topological rigidity on non-collapsed spaces]
								For all $\epsilon>0$, $N \in \mathbb{N}$, and $0<\tau<1$, there exists $\delta=\delta(\epsilon, N, \tau)>0$ such that if a non-collapsed ${\sf RCD}(0, N)$ space $(X, {\sf d}, \mathcal{H}^N)$ satisfies
								$$
								{\sf C}(X,  {\sf d}, {\sf m}, p, q, t) \le  (1+\delta) \frac{{\sf M}(p,q)^\frac{N}{2}}{(4\pi t)^{\frac{N}{2}(\frac{1}{p}-\frac{1}{q})}} , \quad \forall t>1
								$$
								for some $p \in [1, \frac{1}{\tau}]$ and some $q \in [1, \infty]$ with $q-p \ge \tau$,  
								then $X$ is homeomorphic to $\mathbb{R}^N$.
							\end{proposition}
							\begin{proof}
								The assumption implies that ${\sf AVR}(X)$ is close to $1$, namely
								$$(1+\delta)^\frac{pq}{p-q}\leq {\sf AVR}(X)\leq   1,$$
								which concludes the proof. 
							\end{proof}

							In connection with providing an explicit value for $\delta$ as above,  we present a fine topological rigidity result for our heat flow estimate on Riemannian manifolds with non-negative Ricci curvature. To state it, we recall the double induction argument of Perelman
							\cite{Perelman}, quantified by Munn \cite{Munn-JGA}.

							Let $n\geq 2$, and for $k\in\{1,\ldots ,n\}$, we denote by $\delta_{k,n}>0$  the
							smallest positive root of the equation
							$$10^{k+2}C_{k,n}(k)s\left(1+\frac{s}{2k}\right)^k=1$$ in 
							$s>0$,  where
							\begin{equation*}
								C_{k,n}(i)= \left\{
								\begin{array}{lll}
									1 & \mbox{if} & i=0, \\
									3+10C_{k,n}(i-1)+(16k)^{n-1}(1+10C_{k,n}(i-1))^n & \mbox{if} & i\in
									\{1,\ldots ,k\}.
								\end{array}%
								\right.
							\end{equation*} Let $h_{k,n}:(0,\delta_{k,n})\to (1,\infty)$ be the smooth, bijective, and increasing function
							given by
							$$h_{k,n}(s)=\left[1-10^{k+2}C_{k,n}(k)s\left(1+\frac{s}{2k}\right)^k\right]^{-1}.$$
							For every $k\in \{1,\ldots ,n\},$ let
							\begin{equation*}
								\alpha_{MP}(k,n)= \left\{
								\begin{array}{lll}
									1-\left[1+\frac{2}{h_{1,n}^{-1}(2)}\right]^{-1} & \mbox{if} & k=1, \\
									1-\left[1+\left( \frac{1+\cdots +\frac{h_{k-1,n}^{-1}(1+\frac{\delta_{k,n}}{2k})}{2(k-1)}}{h_{1,n}^{-1}\left(1+ \cdots +\frac{h_{k-1,n}^{-1}(1+\frac{\delta_{k,n}}{2k})}{2(k-1)}\right)}\right)^n\right]^{-1} & \mbox{if} &
									k\in \{2,\ldots ,n\},
								\end{array}%
								\right.
							\end{equation*}
							be the  {\it Munn--Perelman
								constant},  see Munn \cite[p. 749-750]{Munn-JGA}. Note that $\alpha_{MP}(\cdot,n)\in (0,1)$ is increasing.  
							
							The following result roughly states that if the "deficit" in the hypercontractivity estimate of the heat flow on an $n$-dimensional  Riemannian manifold
							$(M,g)$ with non-negative Ricci curvature is closer and closer to the sharp  Euclidean estimate \eqref{Bakry-relation} (or to the one in Li's rigidity result, see Theorem \ref{Li-rigidity}), formulated in terms of the Munn--Perelman
							constants, the manifold approaches topologically the Euclidean
							space $\mathbb R^n$  more and more, this being quantified by the fact that the homotopy groups of $(M,g)$ vanish.  
							More precisely, if $\pi_k(M)$ stands for the $k^{\rm th}$
							homotopy group  of $(M,g),$ we can state the following result: 
							
							\begin{theorem}[Topological rigidity on Riemannian manifolds]
								\label{Perelman-Munn-theorem} Let $(M,g)$ be an	$n$-dimensional  Riemannian manifold
								$(M,g)$ with non-negative Ricci curvature, $n\geq 2$, endowed with its natural distance function ${\sf d}_g$ and canonical measure $v_g$.  Assume that for some $1 \le p<q \le \infty$ and $K>0$, one has 
								\begin{equation}\label{perelman-munn}
									{\sf C}(M,  {\sf d}_g, v_g, p, q, t) \le   \frac{{\sf M}(p,q)^\frac{N}{2}K^{\frac{1}{q}-\frac{1}{p}}}{(4\pi t)^{\frac{N}{2}(\frac{1}{p}-\frac{1}{q})}} , \quad \forall t>0.
								\end{equation}
								Then the followings statements hold$:$
								\begin{itemize}
									\item[{\rm (i)}]  the order of the fundamental group $\pi_1(M)$ is bounded
									above by $K^{-1}$ $($in particular, if $K>1/2$, then $M$ is simply connected$);$
									\item[{\rm (ii)}] if $ K> \alpha_{MP}(k_0,n)$ for some $k_0\in \{1,\ldots ,n\},$ then
									$\pi_1(M)=\cdots =\pi_{k_0}(M)=0;$
									\item[{\rm (iii)}] if $ K > \alpha_{MP}(n,n)$,  then
									$M$ is contractible.
								\end{itemize}
							\end{theorem}

							\begin{proof} First of all, by \eqref{perelman-munn} and Theorem \ref{main-theorem-RCD}, it follows that 
								\begin{equation}\label{avr-K}
									{\sf AVR}(M)\geq K>0 .
								\end{equation}

								(i) 
								According to Li \cite[Theorem 2]{Li-Annals},  $(M,g)$ has finite fundamental
								group $\pi_1(M)$, and its order is bounded above by ${\sf AVR}(M)^{-1}\leq K^{-1}$.  In particular, if $K>1/2$, then the order of $\pi_1(M)$ is strictly less than $2$ (thus, it should be $1$), which means that $\pi_1(M)$ is trivial, i.e., $M$ is simply connected. 
								
								(ii) By assumption and \eqref{avr-K}, it follows that ${\sf AVR}(M)\geq K>\alpha_{MP}(k_0,n)\geq \cdots \geq \alpha_{MP}(1,n).$ Due to Munn \cite[Theorem 1.2]{Munn-JGA}, we have that
								$\pi_1(M)=\cdots =\pi_{k_0}(M)=0.$
								
								(iii) Similarly as in (ii), we have that $\pi_1(M)= \cdots =\pi_{n}(M)=0.$  Now, Hurewicz's  theorem implies that $M$ is contractible, see also Perelman \cite[Theorem 2]{Perelman}. 
							\end{proof}

							We conclude the paper by  providing some future problems for the interested readers.
							\begin{enumerate}
								\item It is natural to ask whether the rigidity in Theorem \ref{theorem-equality-RCD} is valid or not in the case $q=\infty$.
								\item Concerning Proposition \ref{prop-almost-rig}, can we state that if the optimal constant ${\sf C}$ is close to the optimal upper bound obtained in (\ref{Heat-bound-main}), then the space together with some point is pmGH close to the $N$-Euclidean cone over an ${\sf RCD}(N-2, N-1)$ space?  The difficulty of this problem is related to the \textit{concentration compactness principle}, which has already been observed in Nobili and Violo \cite{Nobili-Violo}.
							\end{enumerate}

							
							\vspace{0.5cm}
							\noindent \textbf{Acknowledgment.} We would like to thank Michel Ledoux for his interest in the manuscript. We are also indebted to S\'andor Kaj\'ant\'o for useful discussions on the subject of the paper.

						\end{document}